\PassOptionsToPackage{svgnames}{xcolor}
\documentclass[a4paper,final]{amsart}

\usepackage[foot]{amsaddr} 

\usepackage[utf8]{inputenc}
\usepackage[T1]{fontenc}
\usepackage{amsmath,amssymb,amsfonts}
\usepackage{enumerate,color,graphicx,algorithm}
\usepackage{cite}
\usepackage{booktabs,dcolumn}
\usepackage{mathtools}
\usepackage{enumitem}
\usepackage{xspace}
\usepackage{pgfplots}
\pgfplotsset{compat=1.16}
\usepackage{subcaption}

\usepackage{etoolbox}
\usepackage{tikz}
\usetikzlibrary{patterns,angles,quotes}
\def\nudge{.2}
\tikzset{axis/.style={ultra thick, -latex, shorten <=-\nudge cm, shorten >=-2*\nudge cm}}
\tikzset{line/.style={thick}}
\usepackage{orcidlink}

\usepackage{xcolor}
\usepackage{changes}

\numberwithin{equation}{section}

\theoremstyle{plain}
\newtheorem{thm}{Theorem}[section]
\newtheorem{lem}[thm]{Lemma}
\newtheorem{prop}[thm]{Proposition}
\newtheorem{cor}[thm]{Corollary}
\newtheorem{rem}[thm]{Remark}

\newcommand{\defas}{:=}
\newcommand{\asdef}{=:}
\newcommand{\dual}[2]{\left\langle{#1},{#2}\right\rangle}
\newcommand{\norm}[1]{\left\|{#1}\right\|}
\newcommand{\sprod}[2]{\left({#1},{#2}\right)}
\newcommand{\ud}{u_d}
\newcommand{\vphi}{\varphi}

\newcommand{\R}{{\mathbb{R}}}

\newcommand{\Res}{\operatorname{Res}}

\newcommand{\snorm}[1]{\left|{#1}\right|}

\newcommand{\mesh}{\mathcal{M}}
\newcommand{\vertices}{\mathcal{V}}

\begin{document}

\title[A~posteriori error analysis for PDE constraint optimization]{A~posteriori error analysis for optimization \\ with PDE constraints}

\begin{abstract}
We consider finite element solutions to optimization problems, where
the state depends on the possibly constrained control through a linear
partial differential equation.  Basing upon a reduced and rescaled
optimality system, we derive a posteriori bounds capturing the
approximation of the state, the adjoint state, the control and the
observation. The upper and lower bounds show a gap, which grows with
decreasing cost or Tikhonov regularization parameter. This growth is
mitigated compared to previous results and  can be countered by
refinement if control and observation involve compact
operators. Numerical results illustrate these properties for model
problems with distributed and boundary control. 
\end{abstract}

\author[F. Gaspoz]{Fernando Gaspoz}
 \address[Fernando Gaspoz]{Universidad Nacional del Litoral,
 Facultad de Ingenier\'\i a Qu\'\i mica,  Santiago del Estero 2829,
  3000 Santa Fe, Argentina.}
 \email{fgaspoz@fiq.unl.edu.ar}
 \author[C. Kreuzer]{Christian Kreuzer}
 \address[Christian Kreuzer\, \orcidlink{0000-0003-2923-4428}]{Technische Universit\"at Dortmund,
 Fakult\"at f\"ur Mathematik, Vogelpothsweg 87,
 44227 Dortmund, Germany.}
 \email{christian.kreuzer@tu-dortmund.de}
 \author[A. Veeser]{Andreas Veeser}
 \address[Andreas Veeser\, \orcidlink{0000-0002-2152-2911}]
 {Dipartimento di Matematica 'F. Enriques',
 Università degli Studi di Milano,
 Via C. Saldini, 50,
 20133 Milano, Italy.}
 \email{andreas.veeser@unimi.it}
 \author[W. Wollner]{Winnifried Wollner}
\address[Winnifried Wollner\, \orcidlink{0000-0002-6571-8043}]{Universit\"at Hamburg, MIN Fakult\"at, Fachbereich Mathematik, Bundesstr. 55, 20146 Hamburg, Germany.}
\email{winnifried.wollner@uni-hamburg.de}

\date{\today}
\maketitle

\section{Introduction}
%
%
A basic example for optimization problems constrained by partial differential equations (PDEs) is 
\begin{equation}
  \label{sim-min}
  \min_{(q,u)\in K \times \mathring{H}^1(\Omega)}
  \frac12 \norm{ u-\ud }_{L^2(\Omega)}^2 + \frac\alpha2 \norm{q}_{L^2(\Omega)}^2
  \quad\text{subject to}\quad
  -\Delta u=f+q~\text{in}~\Omega,
\end{equation}
where \(\Omega\subset \R^d\) is a suitable domain, $\ud$ is the desired state and
we assume box constraints, i.e.\ $K:=\{q\in L^2(\Omega) \mid a\le
q\le b\}$ with $a<b$.  Such problems are ubiquitous in the optimal control of PDEs. They appear also as Tikhonov regularizations of inverse problems. In the former case the parameter \(\alpha>0\) scales the cost of the control, while in the latter case it is the regularization parameter, which may be chosen quite small. Hence, in any result about such problems, the dependencies on $\alpha$ are critical.

This article concerns the a~posteriori error analysis for problems like
\eqref{sim-min}. An a~posteriori error analysis aims at deriving computable quantities
that, ideally, bound a suitable error from above and below. These quantities can
then be used to assess the approximate solution, and as input for an
adaptive strategy refining the mesh. For both applications, close
bounds will be advantageous and thus, in the case of problems like \eqref{sim-min}, their dependencies on $\alpha$ are of particular interest.

To state our main results for the model problem~\eqref{sim-min}, we consider the linear finite element solution, with or without discretized control; cf.\ \cite{Hinze:2005}. We denote its error by $\mathsf{err}$, which suitably combines the \(\mathring{H}^1\)-errors of state and adjoint state and the \(L^2\)-errors of control and observation. Furthermore, let $\eta^2=\sum_{z \in \vertices} \eta_z^2$ be an estimator, whose local indicators $\eta_z$ quantify the local residuals in Theorem~\ref{T:Apost1-ExampleCC} below in the spirit of \cite{KreuzerVeeser:2021, BonCanNocVee:2024}. Then Theorems~\ref{T:Apost1-ExampleCC} and \ref{T:apost-ExampleCC} imply, as $\alpha \to 0$,
\begin{equation}
  \label{equivalence-for-sim-min}
  c_\Delta \eta
  \le
  \mathsf{err}
  \le
  C_\Delta \min \left\{
    1 + O\left( \frac{1}{\sqrt{\alpha}} \right), \,
    1 + O\left( \frac{1}{\alpha} \right) \frac{\left(\sum_{z\in\vertices} h_z^2 \eta_z^2\right)^{\tfrac{1}{2}}}{\eta} 
  \right\}
  \eta.
\end{equation}
Here $h_z$ denotes the local meshsize around a vertex $z \in
\vertices$ and the constants $c_\Delta$, $C_\Delta$ are independent of
$\alpha$. More precisely, they depend on the technique quantifying
the local residuals, the shape regularity of the mesh cells, and the
Poisson problem, which is associated with the state
equation. Furthermore, the two constants ensure the equivalence 
\begin{equation} 
  \label{equivalence-for-Poisson}
  c_\Delta \eta_\Delta
  \le
  \norm{ \nabla(v - V) }_{L^2(\Omega)}
  \le
  C_\Delta \eta_\Delta,
\end{equation}
where $\eta_\Delta$ is the counterpart of $\eta$ for the $\mathring{H}^1$-error
between any solution $v$ to the Poisson problem and its linear finite
element approximation $V$.

For equivalences like \eqref{equivalence-for-Poisson}, the ratio $C_\Delta/c_\Delta \geq 1$ of the involved constants provides information about the quality of the estimator $\eta_\Delta$. The structure of \eqref{equivalence-for-sim-min} is slightly more involved because the min depends on the indicators $\eta_z$, $z \in \vertices$, and therefore is no part of the constant. Nevertheless, we can measure the quality of the error quantification~\eqref{equivalence-for-sim-min} similarly by the 
ratio of upper to lower bound, which we call gap. 
In this context, the
error bounds~\eqref{equivalence-for-sim-min} lead to
the two following interrelated conclusions: 
\begin{itemize}
\item Taking \eqref{equivalence-for-Poisson} as a benchmark for the
  used estimation technique, the gap in \eqref{equivalence-for-sim-min} becomes
  the one associated with $\eta_\Delta$ for $h \defas \max_{z\in\vertices} h_z \to 0$
  and fixed $\alpha>0$. 
\item Ensuring $\sum_{z\in\vertices} h_z^2 \eta_z^2 \leq \alpha^2
  \eta^2$, the gap in the error bounds \eqref{equivalence-for-sim-min} remains uniformly bounded for  all $\alpha>0$. 
\end{itemize}

Apart from these two features, the error
quantification~\eqref{equivalence-for-sim-min} improves existing  
results. To be more precise, we compare with \cite{KohlsRoeschSiebert:2014,KohlsKreuzerRoeschSiebert:18}, which is the first
abstract a posteriori analysis 
and reviews previous  a posteriori bounds. Therein, as $\alpha \to 0$, the gap grows with
\(O(1 / \alpha)\) if the control is not discretized, otherwise with \(O( 1
/ \alpha^2 )\); see Remark~\ref{R:Comparison-with-KRS} for more
details. In both cases, the error bounds
\eqref{equivalence-for-sim-min} mitigate those growths to
$O(1/\sqrt{\alpha})$. 

The gap in~\eqref{equivalence-for-sim-min} is
closely related to the a~priori results in our previous
article~\cite{GaKrVeWo:2020}. To see this, let us suppose a
variational discretization \cite{Hinze:2005} for simplicity and denote
by $\textsf{best-err}$ the best approximation error in the underlying discrete
spaces. Then \cite[Section~5.2]{GaKrVeWo:2020} ensures the following
variant of C\'{e}a's lemma:  
\begin{equation}
  \label{gen-Cea-lemma}
  \mathsf{err}
  \le
  \mu_\Delta  \min\left\{
    1 + O\left(\frac1{\sqrt{\alpha}}\right), \,
    1 + O\left(\frac{h}{\alpha}\right)
  \right\}
  \textsf{best-err}
  \quad
  \text{as $\alpha,h \to 0$,}
\end{equation}
where $\mu_\Delta$ is the near-best approximation constant associated
with the discretization of the state equation. Inspecting C\'ea's lemma and the
derivation of \eqref{equivalence-for-Poisson}, we see that $\mu_\Delta$ corresponds to the ratio $C_\Delta / c_\Delta$. Furthermore, we observe that the $\min$ in
\eqref{gen-Cea-lemma} is an upper bound for the one
in~\eqref{equivalence-for-sim-min}. 

The similarities between \eqref{equivalence-for-sim-min} and
\eqref{gen-Cea-lemma} result from a `duality' in the respective
derivations. To illustrate this, we outline the key ingredients of
both derivations.  
In both cases, we first consider the possibly nonlinear optimality system
\begin{equation}
  \label{Eq:OptimalControlSystem}
  -\Delta u - \Pi_K\left( -\tfrac{1}{\alpha} p \right) = f,
  \qquad
  -\Delta p - u = -u_d,  
\end{equation}
where the control is implicitly given by \(q=\Pi_K(-\tfrac1\alpha
p)\), and then divide the adjoint equation by $\sqrt{\alpha}$ and
replace the adjoint state $p$ by $z = p/{\sqrt{\alpha}}$.  The latter
has the effect that the perturbations of the Laplacian in both
equations scale like $1/\sqrt{\alpha}$.  Analyzing the properties of
the possibly nonlinear operator $\mathcal{B}_\alpha$ associated with the resulting system prepares the ground for both 
\eqref{equivalence-for-sim-min} and 
\eqref{gen-Cea-lemma}.

More precisely, the continuity and generalized coercivity properties
of $\mathcal{B}_\alpha$ are, respectively, crucial ingredients for the lower
bound and for the first option in the upper bound of
\eqref{equivalence-for-sim-min}. For the first option in
\eqref{gen-Cea-lemma}, one combines the continuity properties of
$\mathcal{B}_\alpha$ with the coercivity properties of its discretization. 

The second options instead hinge on the compactness of the
perturbations of the Laplacian in \eqref{Eq:OptimalControlSystem},
which arises from the embedding $\mathring{H}^1(\Omega) \subset
L^2(\Omega)$. In both cases, the error is decomposed into a main part
and a `compact' part. To this end, we use an auxiliary function in the
spirit of the elliptic reconstruction \cite{MakridakisNochetto:2003}
for \eqref{equivalence-for-sim-min} and a generalized Ritz projection
for \eqref{gen-Cea-lemma}; see also Remark~\ref{R:connections}. The
actual benefit from the available compactness is limited by regularity
theorems for the Poisson equation on polygonal domains and/or
properties of the discretization.

To conclude this introduction, the following remarks are in order:
\begin{itemize}
\item To quantify the local residuals in
  Theorem~\ref{T:Apost1-ExampleCC}, one can use classical techniques,
  see, e.g. \cite{AinsworthOden:2000,Verfuerth:2013}, instead of
  \cite{KreuzerVeeser:2021, BonCanNocVee:2024}. In this case, the
  error quantification \eqref{equivalence-for-sim-min} holds then only up to
  so-called oscillation terms. 
\item In the absence of control constraints, Theorem~\ref{T:apostNCC} 
  provides a variant of \eqref{equivalence-for-sim-min} with an
  improved gap.
\item Although our approach is based upon the reduced optimality
  system \eqref{Eq:OptimalControlSystem}, it is not restricted to
  variational control discretizations as in \cite{Hinze:2005} and covers also
  discretized controls and bounds for their error; see Corollary~\ref{C:gControl}. 
\end{itemize}

This article is organized as follows. Section~\ref{sec:2} recalls the
continuity and coercivity properties of $\mathcal{B}_\alpha$ from
\cite{GaKrVeWo:2020}, along with their proofs due to their importance
and for the sake of a self-contained presentation. The principal part of our a~posteriori
analysis is then developed in Section~\ref{sec:abstract-apost:res=err}. Finally,
Section~\ref{sec:applications} illustrates the obtained results by
applying them to \eqref{sim-min} and a Neumann boundary control
problem, as well as by numerical tests in both cases.

\section{Optimization problem and reduced optimality system}
\label{sec:2}
This section presents the abstract optimization problem to be considered and recalls from \cite{GaKrVeWo:2020} its reduced and rescaled optimality system, along with key ingredients for its well-posedness. These ingredients are also crucial for the subsequent a posteriori error analysis.

To introduce the abstract optimization problem, we take the viewpoint of an optimal control problem and start with the \emph{state equation}. We assume that the control variable \(q\) is taken from a real Hilbert space \((Q, (\cdot,\cdot)_Q)\) with induced norm \(\|\cdot\|_Q\). The relation between control \(q\in Q\) and state \(u\in V_1\) is given by a linear boundary value
problem of the form
\begin{align}\label{bvp}
  Au=f+Cq, 
\end{align}
with the following properties:
\begin{itemize}
\item The differential operator \(A:V_1\to V_2^*\) is defined between
  the \emph{state space} \(V_1\), which is a Hilbert space with scalar
  product \((\cdot,\cdot)_1\) and the
  dual \(V_2^*\) of a second Hilbert space \((V_2,
  (\cdot,\cdot)_2)\). For  \(i=1,2\), the induced norms on \(V_i\) and \(V_i^*\) and the dual pairing are denoted by \(\|\cdot\|_i\), \(\|\cdot\|_{i,*}\),
  and \(\langle\cdot,\cdot\rangle_i\), respectively. We assume that
  the operator A is a linear isomorphism, i.e.\ the
  bilinear form \(a:V_1\times V_2\to \R\) defined by
  \((v_1,v_2)\mapsto \langle Av_1,v_2\rangle_2\) satisfies
  \begin{subequations}
    \label{ass-on-a}
    \begin{gather}
      \label{cont-a}
      M_a
      \defas
      \sup_{\norm{v_1}_1=1}\,\sup_{ \norm{v_2}_{2}=1} a(v_1,v_2) < \infty,
      \\
      \label{non-deg-a}
      \forall v_1 \in V_1
      \quad
      \Big( \forall v_2\in V_1 \; a(v_1,v_2) = 0 \Big) \implies v_1 =0,
      \\
      \label{inf-sup-a}
      m_a
      \defas
      \adjustlimits{\inf}_{\norm{v_2}_2=1}{\sup}_{\norm{v_1}_{1}=1} a(v_1,v_2) > 0;
    \end{gather} 
  \end{subequations}
 compare, e.g., with \cite{Necas:62}.
\item The operator \(C:Q\to
  V_2^*\) is linear and bounded with constant \(M_C\).
\item The load term satisfies \(f\in V_2^*\).
\end{itemize}


Our goal is then to numerically solve the constrained optimization 
problem
\begin{equation}
  \label{min}
  \min_{(q,u)\in K \times V_1}
  \frac12\norm{Iu-\ud }_W^2 + \frac\alpha2 \norm{q}_Q^2
  \quad\text{subject to}\quad
  Au = f+Cq,
\end{equation}
and we suppose in addition:
\begin{itemize}
\item The set \(K\subset Q\) of \emph{admissible controls} is nonempty, closed and convex. 
\item The \emph{cost of the control} is scaled with a parameter
  \(\alpha>0\), which can also be viewed as a Tikhonov regularization. 
\item The \emph{desired state} \(u_d\) lies in the \emph{target
    space} \(W\), which is a  Hilbert space with scalar product
  \((\cdot,\cdot)_W\) and induced norm \(\|\cdot\|_W\).
\item The \emph{observation operator} \(I:V_1\to W\) is linear and
  bounded with constant \(M_I\). 
\end{itemize}

Problem~\eqref{min} is  a quadratic minimization problem with a possibly non-linear constraint (in the case when \(K\neq Q\)). 
As the set of admissible controls is convex and closed, 
standard arguments ensure the existence of a unique solution; 
see, e.g.,~\cite[Theorem~1.1]{Lions:1971} 
or~\cite[Chapter~2.5]{Troeltzsch:2010}.

To formulate the optimality system for~\eqref{min}, we introduce the
adjoint operators 
$A^*$, $C^*$, $I^*$ of $A$, $C$, $I$ 
by
\[
  A^* v_2 = a(\cdot,v_2),
  \quad
  \sprod{q}{C^* v_2}_Q = \dual{Cq}{v_2}_2,
  \quad
  \dual{I^* w}{v_1}_1 = \sprod{Iv}{w}_W
\]
for all $v_1\in V_1$, \(v_2\in V_2\), $q\in Q$, $w\in W$.  
The unique solution $(q,u)$
of~\eqref{min} is equivalently characterized by the existence of a $p \in V_2$, called \emph{adjoint state},  such that the following system
of optimality conditions is satisfied:
\begin{equation}
  \label{mod-opt-pre}
  Au = f+Cq,
  \quad
  A^*p = I^*(Iu-\ud),
  \quad
  q = \Pi_{K}(-\tfrac{1}{\alpha} C^*p);
\end{equation}
compare with \cite{Troeltzsch:2010}. 
Here $\Pi_{K}\mid Q \rightarrow K\subset Q$ is the projection
operator onto the 
admissible set $K$, which is characterized by
\(\|q-\Pi_Kq\|_{Q}=\inf_{p\in K}\|q-p\|_Q\) or, equivalently, by
\begin{align}
  \label{eq:PiK}
  \forall p\in K\quad (q-\Pi_Kq,\Pi_Kq-p)_Q\ge 0.
\end{align}
We notice that $\Pi_{K}$ is Lipschitz continuous with constant $1$ and 
satisfies the inequality
\begin{equation}
  \label{eq:PK-lip}
  (\Pi_{K}(q_1) - \Pi_{K}(q_2), q_1-q_2)_Q
  \ge
  \|\Pi_{K}(q_1)-\Pi_{K}(q_2)\|_Q^2.
\end{equation}
If
\(K=\{q\in L^2(\Omega)\mid a\le q\le b\}\), 
$ V_1 = V_2 = \mathring{H}^1(\Omega) $,
$ A = -\Delta $ is the weak Laplacian,
$ Q = L^2(\Omega) = W$,
$C$ and $I$  are the canonical compact immersions $L^2(\Omega) \to H^{-1}(\Omega) $ and  $H^1_0(\Omega) \to L^2(\Omega)$,
then~\eqref{min} simplifies to the optimization 
problem~\eqref{sim-min} in the introduction.  Notice that, in this
case, \(\Pi_Kq=\max\{\min\{b,q\},a\}\) and the operators 
$C$ and $I$ are related by $C^*=I$.

As in \cite{GaKrVeWo:2020}, we rescale the adjoint variable by
\begin{equation}
  \label{rescaling-of-adjoint-var}
  z = \frac1{\sqrt \alpha}p \in V_2.
\end{equation}
This will turn out advantageous when considering the limit of vanishing Tikhonov regularization; see in particular Remark~\ref{R:Comparison-with-KRS}. We thus obtain the rescaled system
\begin{equation}
  \label{mod-opt}
  Au = f+Cq,
  \quad
  A^*z = \frac1{\sqrt\alpha}I^*(Iu-\ud),
  \quad
  q = \Pi_{K}(-\tfrac{1}{\sqrt \alpha} C^*z),
\end{equation}
and inserting the last equation into the first one, we end up with the
reduced optimality system
\begin{align}
  \label{mod-ropt}
  \begin{pmatrix}
    -\tfrac{1}{\sqrt{\alpha}}  I^* I &  A^*
    \\
    A &-C\Pi_K(-\tfrac{1}{\sqrt{\alpha}}  C^* \cdot)
  \end{pmatrix}
        \begin{pmatrix}
          u \\
          z
        \end{pmatrix}
  =
  \begin{pmatrix}
    - \tfrac{1}{\sqrt{\alpha}}  I^* \ud
    \\
    f
  \end{pmatrix}.
\end{align}
In the typical case, when the operators \(C\) and \(I\) are compact
(see~\eqref{sim-min}), we observe that the operator on the left-hand side of~\eqref{mod-ropt} is a compact perturbation of the control to
state operator \(A\) and its adjoint \(A^*\). This was exploited in
\cite{GaKrVeWo:2020} to show that the near-best approximation constant of
Galerkin approximations for the system~\eqref{mod-ropt} asymptotically
tends to the near-best approximation constant of the Galerkin approximation
of the state equation; cf.\ \eqref{gen-Cea-lemma}. In this work, we aim to exploit this
observation in the a posteriori analysis for~\eqref{mod-ropt}.

The variational formulation of~\eqref{mod-ropt} reads 
\begin{subequations}
  \label{mod-vred-rows}
  \begin{align}
    &\forall \vphi_1\in V_1
    & a(\vphi_1,z) - \tfrac{1}{\sqrt{\alpha}} \sprod{Iu}{I\vphi_1}_W
    &=
      -\tfrac{1}{\sqrt{\alpha}} \sprod{\ud}{I\vphi_1}_W,
    \\
    &\forall \vphi_2\in V_2
    &a(u,\vphi_2) - \sprod{\Pi_K(- \tfrac{1}{\sqrt{\alpha}} C^*z)}{C^*\vphi_2}_Q
    &=\langle f,\vphi_2\rangle_2
      .
  \end{align}
\end{subequations}

This suggests to introduce the Hilbert space 
\begin{equation}
\label{Hilbert-product}
  V \defas V_1 \times V_2
  \quad\text{with}\quad
  \norm{v}
  \defas
  \left(
    \norm{v_1}_1^2 + \norm{v_2}_2^2
  \right)^{\frac12}, \quad v=(v_1,v_2)\in V  
\end{equation}
with dual space \(V^*=V_1^*\times V_2^*\) and induced dual norm \(\|\cdot\|_*\), as well as
the form $b_\alpha\mid V\times V\to\R$ given by
\begin{subequations}
  \label{mod-blin}
  \begin{align}
    b_\alpha(v,\vphi) &\defas \mathbf{a}(v,\vphi)+c_\alpha(v,\vphi) 
  \end{align}
  where
  \begin{align}\label{mod-blina}
    \mathbf{a}(v,\vphi)&\defas a(v_1,\vphi_2)+a(\vphi_1,v_2)
    \\\label{mod-blinc}
    c_\alpha(v,\vphi)&\defas - \left(\Pi_{K}\left(-\tfrac{1}{\sqrt\alpha} C^*v_2\right),C^*\vphi_2\right)_Q
                           -\tfrac{1}{\sqrt\alpha}\left(I v_1,I\vphi_1\right)_W 
  \end{align}
\end{subequations}
for $v=(v_1,v_2), \vphi=(\vphi_1,\vphi_2)\in V$. Although $\mathbf{a}$ is bilinear, $b_\alpha$ is in general only linear in the second argument due to the presence of $\Pi_K$ in $c_\alpha$. In the introduction, we have mentioned the operator $\mathcal{B}_\alpha:V \to V^*$ given by $\mathcal{B}_\alpha v \defas b_\alpha(v,\cdot)$. The bilinear form
$\mathbf{a}:V\times V\to\R$ inherits its continuity and nondegeneracy properties from $a$. More precisely, we have
\begin{equation}
  \label{cont-inf-sup-vec-a}
  \adjustlimits{\sup}_{\norm{v}=1}{\sup}_{\norm{\vphi}=1} |\mathbf{a}(v,\vphi)|
  =
  M_a
  \quad\text{and}\quad
  \adjustlimits{\inf}_{\norm{v}=1}{\sup}_{\norm{\vphi}=1} \mathbf{a}(v,\vphi)
  =
  m_a
\end{equation}
with $M_a$ and $m_a$ from \eqref{ass-on-a}. While the first identity is straight-forward, the second one hinges on the inf-sup-duality, cf.\ Babu\v{s}ka~\cite{Babuska:71},
\begin{equation}
  \label{inf-sup-duality}
  \adjustlimits{\inf}_{\norm{v_1}_1=1}{\sup}_{\norm{\vphi_2}_{2}=1} \mathbf{a}(v_1,\vphi_2)
  =
  \adjustlimits{\inf}_{\norm{v_2}_2=1}{\sup}_{\norm{\vphi_1}_{1}=1} \mathbf{a}(\vphi_1,v_2).
\end{equation}

In this notation,~\eqref{mod-ropt} reads 
\begin{equation}
  \label{vred}
  \begin{aligned}
    \text{find } x\in V &\text{ such that }
    ~\forall \vphi\in V
    \quad
    b_\alpha(x,\vphi) =\langle f,\vphi_2\rangle_{2}
    -\tfrac{1}{\sqrt{\alpha}} \sprod{\ud}{I\vphi_1}_W.
  \end{aligned}
\end{equation}

A pair $x = (u,z) \in V$ solves the variational formulation~\eqref{vred} of the reduced and rescaled optimality system if and only if  the 
triple $\big(u,z,\Pi_{K}\left(-C^*z/\sqrt{\alpha}\right)\big)\in V\times Q$ satisfies the optimality 
system~\eqref{mod-opt}.  Consequently, thanks to the convexity 
of~\eqref{min}, $x=(u,z)\in X$ is a solution of~\eqref{vred} if and
only if 
$\big(\Pi_{K}\left(-C^*z/\sqrt{\alpha}\right),u\big)\in Q \times V_1$ is a solution of the optimization 
problem~\eqref{min}.

Although $b_\alpha$ is not bilinear in general, we have derived in~\cite[Theorem~5.1]{GaKrVeWo:2020}  properties that generalize the continuity and inf-sup stability of bilinear forms. As this is fundamental for the following a~posteriori analysis, we state them and repeat their proofs.  To this end, we
introduce on \(V\) the seminorm
\begin{equation}
  \label{enorm}
  \snorm{v}
  \defas
  \left( \norm{Iv_1}_W^2 + \norm{C^*v_2}_Q^2
  \right)^{1/2}
\end{equation}
and its relative the pseudometric
\begin{equation}\label{eq:deltaK}
  \delta_\alpha(v,w)^2
  \defas
  \alpha \norm{\Pi_{K}\left (-\tfrac{1}{\sqrt{\alpha}}C^*v_2 \right)
  -
  \Pi_{K}\left( -\tfrac{1}{\sqrt{\alpha}}C^*w_2 \right) }_Q^2
  +
  \norm{I(v_1-w_1)}_W^2.
\end{equation}
For \(v,w\in V\) choosing \(\vphi=(-(v_1-w_1), v_2-w_2)\), we
have
\begin{subequations}
  \label{cont-coercivity-c-cc}
  \begin{equation}
    c_\alpha(v,\vphi) - c_\alpha(w,\vphi)
    \geq
    \frac{1}{\sqrt{\alpha}} \delta_\alpha(v, w)^2,
  \end{equation}
  while, for any $v,w,\vphi \in V$, we have
  \begin{equation}
    \label{cont-c-cc}
    |c_\alpha(v,\vphi) - c_\alpha(w, \vphi)|
    \leq
    \frac{1}{\sqrt{\alpha}} \delta_\alpha(v,w) \snorm{\vphi}
  \end{equation}
\end{subequations}
and, thanks to~\eqref{eq:PK-lip},
\begin{equation}
  \label{delta<enorm}
  \delta_\alpha(v,w)
  \leq
  \snorm{v-w}.
\end{equation}
The continuity bound \eqref{cont-c-cc} and a Cauchy-Schwarz inequality lead to
\begin{equation}
  \label{cont-cc}
  |b_\alpha(v,\vphi) - b_\alpha(w,\vphi)|
  \leq M_a\,
  d_\alpha(v,w) \norm{\vphi},
\end{equation}
where the metric \(d_\alpha\) is defined by
\begin{equation}\label{eq:dK}
  d_\alpha(v,w)
  \defas
  \norm{v-w} + \frac1{\sqrt{\alpha}} \frac{M}{M_a} \delta_\alpha(v,w),
  \quad
  v,w \in V,
\end{equation}
with
\begin{equation*}
  M \defas \max\{M_I,M_C\},
\end{equation*}
and $\norm{\cdot}$ is from \eqref{Hilbert-product}.
This brings us in the position to state and prove the announced properties of the operator associated with the reduced and rescaled optimality system~\eqref{mod-ropt}.

\begin{thm}[Continuity and inf-sup stability of form $b_\alpha$]
  \label{T:bK}
  For any $v,w,\vphi\in V$, we have 
  \begin{subequations}
    \begin{align}\label{bK:Cont}
      | b_\alpha(v,\vphi) - b_\alpha(w,\vphi) |
      &\leq
        M_a \,d_\alpha(v,w) \norm{\vphi}
        \intertext{and there exists \(0\neq\psi\in V\) such that}
        b_\alpha \big( v,\psi) - b_\alpha \big( w,\psi  \big) \label{bK:infsup}
      &\geq
        \frac{m_a}{\kappa} \,d_\alpha(v,w) \norm{\psi},
    \end{align}
  \end{subequations}
  where $\kappa$ is defined by
  \begin{equation}
    \label{kappa}
    \kappa
    =
    \frac{1+2L}{1+L} \left(
      1 + \frac{M}{m_a} \left( 1 + 2L
      \right)
    \right)
  \quad\text{with}\quad
    L=\frac{M}{\sqrt\alpha}.
  \end{equation}
\end{thm}

\begin{proof}
The first inequality is~\eqref{cont-cc}.  In order to prove the second one, we choose, for fixed $v,w\in V$, 
  \begin{equation*}
    \psi
    =
    m_a(A^{-1}J_2(v_2-w_2), A^{-*}J_1(v_1-w_1)) + \gamma (-(v_1-w_1),(v_2-w_2)),
  \end{equation*}
  for some \(\gamma>0\) to be determined later.
  Here the evaluation maps \(J_i:V_i\to V_i^*\) are defined by
  \(\langle J_i\psi_i,\cdot\rangle_i\defas(\psi_i,\cdot)_i\) for
  \(i=1,2\). With this choice, we obtain
  \begin{multline*}
    b_\alpha(v,\psi)-b_\alpha(w,\psi)\\
    \begin{aligned}
      &=m_a\left( a(v_1-w_1, A^{-*}J_1(v_1-w_1))+a(A^{-1}J_2(v_2-w_2),
        v_2-w_2)\right)
      \\
      &\quad - m_a\left(\Pi_{K}\left(-\tfrac{1}{\sqrt\alpha}
          C^*v_2\right)-\Pi_{K}\left(-\tfrac{1}{\sqrt\alpha}
          C^*w_2\right),C^*A^{-*}J_1(v_1-w_1)\right)_Q
      \\
      &\quad-m_a\tfrac{1}{\sqrt\alpha}\left(I (v_1-w_1),I
        A^{-1}J_2(v_2-w_2)\right)_W
      \\
      &\quad - \gamma \,m_a\left(\Pi_{K}\left(-\tfrac{1}{\sqrt\alpha}
          C^*v_2\right)-\Pi_{K}\left(-\tfrac{1}{\sqrt\alpha}
          C^*w_2\right),C^*(v_2-w_2)\right)_Q
      \\
      &\quad+\gamma\tfrac{1}{\sqrt\alpha}\left(I (v_1-w_1),I
        (v_1-w_1)\right)_W
      \\
      &\ge m_a\|v-w\|^2 - \frac{M}{\sqrt\alpha}
      \,\delta_\alpha(v,w)\|v-w\|+\frac{\gamma}{\sqrt\alpha}\,\delta_\alpha(v,w)^2
      \\
      &\ge
      m_a\left(\|v-w\|+\frac{M}{M_a}\frac1{\sqrt\alpha}\,\delta_\alpha(v,w)\right)\|v-w\|
      -\frac{2M}{\sqrt\alpha}
      \,\delta_\alpha(v,w)\|v-w\|
      \\
      &\quad +\frac{\gamma}{\sqrt\alpha}\delta_\alpha(v,w)^2, 
    \end{aligned}
  \end{multline*}
  where we used \(m_a\le M_a\) as well as the continuity of \(C\) and
  \(I\). Using Young's inequality \(2st\le \epsilon s^2+t^2/\epsilon\)
  with \(\epsilon=\frac{L}{1+2L}m_a>0\), we may bound the critical term
  by
  \begin{align*}
    \frac{2M}{\sqrt\alpha}
    \,\delta_\alpha(v,w)\|v-w\|\le
    \frac{L}{1+2L}m_a\,\|v-w\|^2+\frac{1+2L}{L}\frac{M^2}{m_a\alpha} \delta_\alpha(v,w)^2.
  \end{align*}
  Consequently, choosing
  \begin{align*}
    \gamma=\frac{M}{m_a}\left(1+2L\right),
  \end{align*}
  we arrive at
  \begin{align*}
    \begin{aligned}
      b_\alpha(v,\psi)-b_\alpha(w,\psi)&\ge\frac{1+L}{1+2L}\,m_a\,d_\alpha(v,w)\|v-w\|
      \\
      &\ge \frac1{\kappa} \,m_a\,d_\alpha(v,w)\|\psi\|.
    \end{aligned}
  \end{align*}
  Here the last inequality follows from
  \begin{equation*}
    \|\psi\|\le \left(1+\frac{M}{m_a}(1+2L)\right)\|v-w\|.
    \qedhere
  \end{equation*}
\end{proof}

\section{Relating error and residual}
\label{sec:abstract-apost:res=err}
%
%
This section constitutes the principal part of our a posteriori analysis for the abstract optimal control problem \eqref{min}. A typical approach to such an analysis can be subdivided into the following three steps; cf., e.g., \cite[\S4]{BonCanNocVee:2024} or \cite[\S1.4]{Verfuerth:2013}: given an approximate solution, 
\begin{itemize}
\item relate the error (norm) to a suitable norm of the  so-called residual, a quantity that depends only on data and the approximate solution,
\item split the residual norm, which is typically of dual character, into local contributions,
\item  further split the local contributions into a computable part involving the approximate solution and an oscillatory part depending only on data.
\end{itemize}
This section addresses the first step. It then turns out that the
following two steps hinge only on the particular structure of the state and adjoint equations. Therefore, they are not addressed in general and  postponed to the applications in Section~\ref{sec:applications}.

We shall base our a~posteriori analysis of the optimization problem \eqref{min} on the variational formulation \eqref{vred} of the reduced and rescaled optimality system \eqref{mod-ropt}. Let \(\tilde x=(\tilde u,\tilde z)\in V\) be some approximation of $x=(u,z)$, where $u$ is the exact state and $z$ the (rescaled) exact adjoint state. We define the \emph{residual} in $\tilde{x}$ by
\begin{equation*}
  \Res(\tilde{x})
  \defas
  \begin{pmatrix}
    - \tfrac{1}{\sqrt{\alpha}}  I^* \ud
    \\
    f
  \end{pmatrix}
  -
  \begin{pmatrix}
    -\tfrac{1}{\sqrt{\alpha}}  I^* I &  A^*
    \\
    A &-C\Pi_K(-\tfrac{1}{\sqrt{\alpha}}  C^* \cdot)
  \end{pmatrix}
  \begin{pmatrix}
    \tilde{u} \\
    \tilde{z}
  \end{pmatrix},
\end{equation*}
or, equivalently in variational form, by
\begin{equation*}
  \langle \Res(\tilde x),\vphi\rangle
  =
  \left\langle
    f - A\tilde u + C \Pi_K\left(-\tfrac1{\sqrt\alpha}C^*\tilde z\right), \vphi_2
  \right\rangle_2
  +
  \left\langle
    \tfrac1{\sqrt\alpha}I^*(I\tilde{u}-u_d) - A^*\tilde z, \vphi_1
  \right\rangle_1.
\end{equation*}

In what follows, we shall offer three approaches to relate the  error to the residual $\Res(\tilde{x})$, strengthening  the relationship under increasingly stronger assumptions.
For comparison, it is useful to recall that the assumptions on the state equation imply the following error-residual relationship: if $v \in V_1$ verifies $Av=g$ and $\tilde{v} \in V_1$ is some approximation of $v$, then error and residual norm $\| g - A\tilde{v} \|_{2,*}$ are equivalent:
\begin{align}
  \label{eq:err=ResA}
  \frac1{M_a}\| g-A\tilde{v}\|_{2,*}
  \le
  \|v-\tilde{v}\|_1
  \le
  \frac1{m_a}\| g-A\tilde{v}\|_{2,*}.
\end{align}
(The proof of this follows the lines of the proof of
Theorem~\ref{T:Apost-1} below, replacing the form $b_\alpha$ by the
bilinear form $a$.) Notice that there is a \emph{gap} between the upper and
lower bound for $M_a \gg m_a$, which can be
measured by the ratio $M_a/m_a \geq 1$ of upper to lower bound.

\subsection{Using continuity of control and observation}
\label{sec:direct-error=res}
%
In contrast to the subsequent subsections, here we shall
not assume compactness of  the control operator $C$ and the observation
operator $I$ in addition to the conditions in~Section~\ref{sec:2}, i.e.\ they are just linear and bounded operators.

The residual and error are related through a possibly nonlinear operator. In fact,  since $x=(u,z)$ is the exact solution of \eqref{mod-ropt}, we have the identity
\begin{multline*}
  \Res \begin{pmatrix}
    \tilde{u} \\ \tilde{z}
  \end{pmatrix} 
  =
  \begin{pmatrix}
    -\tfrac{1}{\sqrt{\alpha}}  I^* I &  A^*
    \\
    A &-C\Pi_K(-\tfrac{1}{\sqrt{\alpha}}  C^* \cdot)
  \end{pmatrix}
  \begin{pmatrix}
    u  \\ z 
  \end{pmatrix}
  \\
  {}	-
  \begin{pmatrix}
    -\tfrac{1}{\sqrt{\alpha}}  I^* I &  A^*
    \\
    A &-C\Pi_K(-\tfrac{1}{\sqrt{\alpha}}  C^* \cdot)
  \end{pmatrix}
  \begin{pmatrix}
    \tilde{u} \\ \tilde{z}
  \end{pmatrix},
\end{multline*}
which in variational form reads
\begin{equation}
  \label{eq:RES}
  \langle \Res(\tilde x),\vphi\rangle
  =
  b_\alpha(x,\vphi)-b_\alpha(\tilde x,\vphi),\quad \vphi\in V.
\end{equation}
Following standard arguments, this identity can be combined  with the
properties of the form $b_\alpha$ in Theorem~\ref{T:bK}.  This \emph{direct
  approach} leads to the following relationship between the residual
in the dual norm $\norm{\cdot}_*$, see \eqref{Hilbert-product}, and the \emph{$d_\alpha$-error}, i.e.\ the distance of the states and their approximations in the metric~\eqref{eq:dK}.

\begin{thm}[Bounding the $d_\alpha$-error -- general case]
\label{T:Apost-1}
Let \(x=(u,z)\in V\) be the solution to~\eqref{vred} and $\tilde{u} \in V_2$ some approximate state and $\tilde{z} \in V_1$ some approximate rescaled adjoint state. Writing \(\tilde x = (\tilde{u}, \tilde{z}) \), their $d_\alpha$-error is equivalent to the dual norm of the residual:
  \begin{align*}
    \frac1{M_a}\left\|\Res(\tilde x)\right\|_{*}
    \le
    d_\alpha(x,\tilde x)
    \le
    \frac{\kappa}{m_a}\left\|\Res(\tilde  x)\right\|_{*}
  \end{align*}
  with $\kappa$ from \eqref{kappa}. 
\end{thm}

\begin{proof}
  The identity~\eqref{eq:RES} and the Lipschitz continuity \eqref{bK:Cont} imply that, for any \(\vphi\in V\), 
  \begin{align*}
    |\langle \Res(\tilde x), \vphi\rangle|
    = 
    |b_\alpha(x,\vphi) - b_\alpha(\tilde x,\vphi)| 
    \leq
    M_a \,d_\alpha(x,\tilde x) \norm{\vphi}.
  \end{align*}
  Dividing by \(\norm{\vphi}\) and taking the supremum over all \(0\neq\vphi\in V\) proves the lower bound.

  For the upper bound, it follows from~\eqref{bK:infsup} and \eqref{eq:RES} that there exists \(0\neq\psi\in V\) such that
  \begin{align*}
    \frac{m_a}{\kappa}d_\alpha(x,\tilde x)\|\psi\|
    \le
    b_\alpha(x,\psi) - b_\alpha(\tilde x,\psi)
    =
    \langle \Res(\tilde x), \psi\rangle
    \le
    \|\Res(\tilde x)\|_* \|\psi\|.
  \end{align*}
  Thus dividing by $\norm{\psi}>0$ finishes the proof.
\end{proof}

\begin{rem}[Gap in bounding $d_\alpha$-error -- general case]
  \label{R:vanish-regularization}
  The upper and lower bound in Theorem~\ref{T:Apost-1} present the gap $(M_a \kappa) /m_a$. With respect to \eqref{eq:err=ResA} concerning the state equation, this gap is amplified by the factor $\kappa$ from Theorem~\ref{T:bK}.

  For the limit  $\alpha\to0$ of the Tikhonov regularization parameter, with $I$ and $C$ fixed, we have $L =\frac{M}{\sqrt{\alpha}}\to \infty$ and, therefore, the amplification factor $\kappa$ verifies 
  \begin{equation}
    \label{kappa;vanish-regularization}
    \kappa
    =
    \left( \frac{4M^2}{m_a} + o(1) \right) \frac{1}{\sqrt{\alpha}}
    \quad\text{as}\quad
    \alpha \to 0.
  \end{equation}
\end{rem}

Theorem~\ref{T:Apost-1} considers only approximations of  the state and the (rescaled) adjoint state. As outlined above, this arises from the direct application of Theorem~\ref{T:bK}, which analyses features of the rescaled and reduced optimality system \eqref{mod-ropt}. The absence of an explicit control in a discretization of \eqref{mod-ropt} then corresponds to assuming that the approximate control is given by
\( \Pi_K\big(-\frac1{\sqrt\alpha}C^* \tilde z\big) \).
Therefore, we can combine Theorem~\ref{T:Apost-1} with suitable triangle inequalities to consider the complete optimality system \eqref{mod-opt} with the \emph{augmented $d_\alpha$-error}
\begin{equation*}
  d_\alpha(x,\tilde x)+\frac{M}{M_a}\left\|\tilde q-q\right\|_Q,
\end{equation*}
where $\tilde{q}$ is some approximation of the exact control
\begin{equation}
  \label{exact-control}
  q \defas \Pi_K \big( -\frac1{\sqrt\alpha}C^* z \big).
\end{equation}
Of course, the bounds then involve an additional residual of the approximate control.

\begin{cor}[Bounding the augmented $d_\alpha$-error -- general case]
\label{C:gControl}
In addition to the assumptions of Theorem~\ref{T:Apost-1} and \eqref{exact-control}, let $\tilde{q} \in K$ be some approximate control. Then the augmented $d_\alpha$-error of $(\tilde{u},\tilde{z},\tilde{q})$ is bounded from above and below by
  \begin{multline*}
    d_\alpha(x,\tilde x)+\frac{M}{M_a}\left\|q- \tilde{q} \right\|_Q
    \eqsim
    \left(
      \left\| f - A\tilde u + C\tilde q \right\|_{2,*}^2
      +
      \left\| A^*\tilde z -\tfrac1{\sqrt\alpha}I^* (I\tilde u-u_d)\right\|_{1,*}^2
    \right)^{\frac12}
    \\
    +\frac{M}{M_a} \left\|
      \tilde q - \Pi_K \left( -\tfrac1{\sqrt\alpha}C^*\tilde  z \right)
    \right\|_Q,
  \end{multline*}
  where the constants hidden in \(\eqsim\) present the same asymptotics for $\alpha\to0$ as those in Theorem~\ref{T:Apost-1}.
\end{cor}

\begin{proof}
  We verify only the lower bound; the upper one follows from similar arguments upon noting \( (\kappa M_a) / m_a \ge 1\). First notice that the triangle inequality ensures
  \begin{align*}
    \left\|f-A\tilde u +C\tilde q\right\|_{2,*}
    \le
    \left\| f-A\tilde u + C \Pi_K\left( -\tfrac1{\sqrt\alpha}C^* \tilde z\right) \right\|_{2,*}
    +
    M \left\| \tilde q-\Pi_K\left(-\tfrac1{\sqrt\alpha}C^* \tilde z\right)\right\|_Q
  \end{align*}
  and, combined with the definition \eqref{exact-control} of $q$,
  \begin{multline*}
    M \left\| \tilde q-\Pi_K\left(-\tfrac1{\sqrt\alpha}C^* \tilde z\right) \right\|_Q
    \\
    \begin{aligned}
      &\le
      M\left\| \tilde q- q\right\|_Q
      +
      M\left\|
        \Pi_K\left(-\tfrac1{\sqrt\alpha}C^* \tilde z\right)
        -\Pi_K\left(-\tfrac1{\sqrt\alpha}C^* z\right)
      \right\|_Q
      \\
      &\le M\left\|\tilde
        q-q\right\|_Q+\frac{M}{\sqrt{\alpha}}\,\delta_\alpha(x,\tilde
      x)\le M\left\|\tilde q-q\right\|_Q+M_a d_\alpha(x,\tilde x) .
    \end{aligned}
  \end{multline*}
  Hence, the lower bound follows from the one in Theorem~\ref{T:Apost-1} without additional appearance of the parameter $\alpha$.
\end{proof}

It is instructive to compare our approach based upon the reduced and rescaled optimality system with the previous a~posteriori analysis in \cite{KohlsRoeschSiebert:2014,KohlsRoeschSiebert:2012}.

\begin{rem}[Comparison with \cite{KohlsRoeschSiebert:2014,KohlsRoeschSiebert:2012}]
  \label{R:Comparison-with-KRS}
  %
  We refer in particular to \cite[Theorem~2.2]{KohlsRoeschSiebert:2014}, the proof of which presents the
  explicit dependencies between error terms and residual. Compared to our approach, there are two main differences
  regarding the error notion: First, in contrast to~\eqref{rescaling-of-adjoint-var}, the adjoint state in
  \cite{KohlsRoeschSiebert:2014} is not rescaled and the same holds true for the residual of the adjoint equation. Second,
  the error notion in~\cite{KohlsRoeschSiebert:2014} does not explicitly involve the
  observation error.

  For $\alpha\to0$, the gap between upper and lower bound in~\cite[Theorem~2.2]{KohlsRoeschSiebert:2014} grows like \(O(1/\alpha^2)\) in general and reduces to $O(1/\alpha)$ if the residual of the approximate control vanishes as in the reduced optimality system. In contrast, the respective gaps in Corollary~\ref{C:gControl} and Theorem~\ref{T:Apost-1} are only \(O(1/\sqrt{\alpha})\). This slower growth arises from the properties of the form $b_\alpha$ in Theorem~\ref{T:bK} about the reduced and rescaled optimality system. The amplified growth in~\cite[Theorem~2.2]{KohlsRoeschSiebert:2014} in the general case results from the fact that the absence of the observation error in the error notion has to be compensated by bounding in terms of the stability of the observation operator and the error of the adjoint state.
\end{rem}

\subsection{Using compactness of control and observation}
\label{sec:apost-ellRecon}
%
In this subsection, we assume that both the control operator $C$ and
the observation operator $I$ are bounded and compact. This situation
is often encountered in applications;  see Section~\ref{sec:applications} for two examples.

Let us first outline the idea how to take advantage of the compactness assumptions. To this end, we may split the operator of the reduced optimality system \eqref{mod-ropt} as follows:
\begin{multline*}
  \begin{pmatrix}
    -\tfrac{1}{\sqrt{\alpha}}  I^* I &  A^*
    \\
    A &-C\Pi_K(-\tfrac{1}{\sqrt{\alpha}}  C^* \cdot)
  \end{pmatrix}
  \\
  =
  \begin{pmatrix}
    0 &  A^*
    \\
    A & 0
  \end{pmatrix}
  +
  \begin{pmatrix}
    -\tfrac{1}{\sqrt{\alpha}}  I^* I &  0
    \\
    0 &-C\Pi_K(-\tfrac{1}{\sqrt{\alpha}}  C^* \cdot)
  \end{pmatrix},
\end{multline*}
where the first term on the right-hand side is $\alpha$-independent,
while the second can be viewed as an $\alpha$-dependent, compact
perturbation. Since the first term is invertible, we can apply this
operator splitting to the residual $\Res(\tilde{x})$ and translate it
into a splitting of the error $x-\tilde{x}$. More precisely, applying
the inverse of the first operator on the right-hand side to the residual, we define
an auxiliary function $R\tilde{x} \in V$ by 
\begin{align}
  \label{eq:MotivationReconstruction}
  R\tilde x-\tilde x
  &
    =\begin{pmatrix}
      0 &  A^{-1}
      \\
      A^{-*} & 0
    \end{pmatrix}\Res(\tilde x)
\end{align}
and write 
\[
  x-\tilde x
  =
  (x-R\tilde x) + (R\tilde x -\tilde x),
\]
where Theorem~\ref{T:Apost-1} and \eqref{eq:err=ResA} together with its `adjoint' variant, respectively, imply
\begin{gather}
  \label{equivalence-for-x-Rtildex}
  \frac{1}{M_a}\norm{\Res(R\tilde{x})}_{*}
  \leq
  d_\alpha(x,R\tilde{x})
  \leq
  \frac{\kappa}{m_a} \norm{\Res(R\tilde{x})}_{*},
  \\ \label{eq:apostr-Recon} 
  \frac{1}{M_a}\norm{\Res(\tilde{x})}_{*}
  \leq
  \| \tilde{x} - R\tilde{x} \|
  \leq
  \frac{1}{m_a} \norm{\Res(\tilde{x})}_{*}. 
\end{gather}
Observe that~\eqref{eq:apostr-Recon} relies only on properties of the
state equation and involves the known residual $\Res\tilde{x}$. In
contrast,~\eqref{equivalence-for-x-Rtildex} involves $\kappa$ and so
the regularization parameter $\alpha$ as well as the unknown residual
$\Res(R\tilde{x})$. However, applying the operator of the optimality
system \eqref{mod-ropt} to the decomposition of the error reveals 
\begin{equation}
  \label{advantage-thanks-to-compactness}
  \Res(R\tilde{x})
  =
  \begin{pmatrix}
    \tfrac{1}{\sqrt{\alpha}}  I^* I &  0
    \\
    0 &C\Pi_K(-\tfrac{1}{\sqrt{\alpha}}  C^* \cdot)
  \end{pmatrix}
  (R\tilde{x}-\tilde{x}).
\end{equation}
Therefore, we can bound the critical error $x-R\tilde{x}$ in terms of the error $\tilde{x}-R\tilde{x}$, which is accessible through the known residual $\Res(\tilde{x})$. More importantly, in view of the compactness properties of the operator matrix in \eqref{advantage-thanks-to-compactness}, we may take a norm of $R\tilde{x}-\tilde{x}$ that is weaker than $\norm{\cdot}$ and, thus, allows for faster convergence.

Let us put  this \emph{approach based on compactness} into the context of previous and related techniques.
\begin{rem}[Connection with \cite{KohlsRoeschSiebert:2014}, elliptic reconstruction, Wheeler's trick, and Schatz's argument]
\label{R:connections}
The above construction of the auxiliary function $R\tilde{x}$ is also implicitly used in \cite{KohlsRoeschSiebert:2014}, but as simple perturbation argument without exploiting the compactness assumptions through \eqref{advantage-thanks-to-compactness}. The use with \eqref{advantage-thanks-to-compactness} is quite similar to the so-called elliptic reconstruction of \cite{MakridakisNochetto:2003, LakkisMakridakis:2006} in the context of parabolic equations.  The error decomposition with the elliptic reconstruction can be viewed as the a~posteriori counterpart of the error decomposition in Wheeler's trick \cite{Wheeler:1973} by means of the Ritz projection. Similarly, but incorporating the role of compactness, the above approach with $R\tilde{x}$ is the a~posteriori counterpart of the a~priori analysis in \cite{GaKrVeWo:2020} based upon Schatz's argument  \cite{Schatz:1974}.
\end{rem}

We proceed in variational terms and, in line with \eqref{eq:MotivationReconstruction},  define the auxiliary operator \(R=(R_1,R_2):V\to V\) for $\tilde{x}=(\tilde{u},\tilde{z})$ by
\begin{subequations}
  \label{eq:Rell}
  \begin{align}
    \label{eq:Rell1}
    \forall \vphi_2\in V_2\qquad a(R_1\tilde x, \vphi_2)
    &=
      \langle f, \vphi_2\rangle_2
      +
      \left(\Pi_K\left(-\tfrac1{\sqrt\alpha}C^* \tilde z\right), C^*\vphi_2\right)_Q,
    \\ \label{eq:Rell2}
    \forall \vphi_1\in V_1\qquad a(\vphi_1,R_2\tilde x)
    &=
      \tfrac1{\sqrt\alpha} \left(I\tilde u -u_d,I\vphi_1\right)_W.
  \end{align}
\end{subequations}
The crucial relationship \eqref{advantage-thanks-to-compactness} then amounts to 
\begin{equation}
  \label{Eq:x-Rtildex}
  \langle\Res(R\tilde x),\vphi\rangle
  =
  b_\alpha(x,\vphi)-b_\alpha(R\tilde x,\vphi)
  =
  c_\alpha(R\tilde x,\vphi) - c_\alpha(x,\vphi),
\end{equation}
whence the critical error part is bounded with the additional help of \eqref{equivalence-for-x-Rtildex} by 
\begin{align}
  \label{Eq:APostReminder}
  \begin{multlined}
    \frac{m_a^2}{\kappa^2}\, d_\alpha(x,R\tilde x)^2
    \le
    \|\Res(R\tilde x)\|_*^2
    \\
    \begin{aligned}
      &=
      \left\|
        C\Pi_K\left(-\frac{C^*R_2\tilde x}{\sqrt\alpha}\right)
        -
        C\Pi_K\left(-\frac{C^*\tilde z}{\sqrt\alpha}\right)\right\|_{2,*}^2
      +
      \left\| \frac1{\sqrt\alpha}I^*I(\tilde u-R_1\tilde x)\right\|_{1,*}^2
      \\
      &\asdef
      \frac{M^2}{\alpha}
      \delta_\alpha^*(R\tilde x,\tilde x)^2.
    \end{aligned}
  \end{multlined}
\end{align}
The approach based on compactness then leads to the following alternative to Theorem~\ref{T:Apost-1} of the direct approach.

\begin{thm}[Bounding the $d_\alpha$-error -- compact case]
\label{T:compact-case}
Let \(x=(u,z)\in V\) be the solution to~\eqref{vred} and $\tilde{u} \in V_2$ be some approximate state and $\tilde{z} \in V_1$ some approximate rescaled adjoint state. Writing \(\tilde x = (\tilde{u}, \tilde{z}) \), their error is bounded from above and below as follows:
  \begin{align*}
    d_\alpha(x,\tilde x)
    \le
    \frac1{m_a} \left(
    1 + \left(\kappa+\frac{m_a}{M_a}\right) \frac{M}{\sqrt\alpha} \,
    \frac{\delta_\alpha^*(R\tilde x,\tilde x)}{\|\Res(\tilde x)\|_*}
    \right) \|\Res(\tilde x)\|_*
  \end{align*}
  and
  \begin{align*}
    \frac1{M_a}\|\Res(\tilde x)\|_*
    \le
    d_\alpha(x,\tilde x).
  \end{align*}
\end{thm}

\begin{proof}
  The lower bound is a repetition from Theorem~\ref{T:Apost-1} and it remains to show the upper bound. The triangle inequality and the definition \eqref{eq:dK} of $d_\alpha$ yield
  \begin{align*}
    d_\alpha(x,\tilde x)
    &\le
      d_\alpha(x,R\tilde x) + d_\alpha(R\tilde x,\tilde x)
    \\
    &=
      d_\alpha(x,R\tilde x)+\|R\tilde x-\tilde x\|
      +
      \frac{M}{M_a} \frac1{\sqrt\alpha} \delta_\alpha^*(R\tilde x,\tilde x).
  \end{align*}
  Consequently, \eqref{Eq:APostReminder} and \eqref{eq:apostr-Recon} ensure the upper bound and the proof is finished.
\end{proof}

The following three remarks elucidate why Theorem~\ref{T:compact-case} is a valid alternative to Theorem~\ref{T:Apost-1} of the direct approach.

\begin{rem}[Using compactness]
\label{R:partial-compactness}
The definition \eqref{Eq:APostReminder} of $\delta^*_\alpha(R\tilde{x},\tilde{x})$ offers all available compactness, given by the operators $C$, $I$ and their
adjoints. Recall that this compactness will allow to choose a norm
for $R\tilde{x}-\tilde{x}$ that is weaker than $\norm{\cdot}$ and
enjoys faster convergence. Such a faster convergence is usually
ensured with the help of a duality argument  depending on regularity
properties of the state and adjoint equation as well as on
approximation properties of the discretization providing the
discrete solution. In particular, it may happen that if we insert
the bound 
\begin{equation*}
    \delta^*_\alpha(R\tilde{x},\tilde{x})
    \leq
    \delta_\alpha(R\tilde{x},\tilde{x})
  \end{equation*}
  in the upper bound of Theorem~\ref{T:compact-case} and exploit only the compactness in the latter term, the faster convergence of $\delta^*_\alpha(R\tilde{x},\tilde{x})$ is already fully captured. See also Remark~\ref{R:limitations} below.
\end{rem}

\begin{rem}[Gap in bounding $d_\alpha$-error -- compact case]
  \label{R:abstractCompact1}
  If we do not exploit the compactness assumptions by inserting the bound
  \begin{align*}
    \delta_\alpha^*(R\tilde x,\tilde x)
    \le
    \delta_\alpha(R\tilde x,\tilde x)
    \le
    M\|\tilde x-R\tilde x\|
  \end{align*}
  in the upper bound of Theorem~\ref{T:compact-case}, the gap between
  its upper and lower bound is
  \(O(\kappa/\sqrt{\alpha})=O(1/\alpha)\) as \(\alpha\to 0\);
  compare with Remark~\ref{R:vanish-regularization}.  This illustrates
  that the upper bound of Theorem~\ref{T:compact-case} improves on the
  one in Theorem~\ref{T:Apost-1} only if
  $\delta^*_\alpha(R\tilde{x},\tilde{x})$ is relatively small with
  respect to $\norm{\Res(\tilde{x})}$. The significance of
  Theorem~\ref{T:compact-case} thus hinges on the aforementioned
  effect of the compactness assumptions. Indeed, if 
  \begin{align*}
    \frac{\delta^*_\alpha(R\tilde x,\tilde x)}
    {\|\tilde x-R\tilde x\|}
    \to
    0
    \qquad\text{for }\tilde  x\to x \text{ and fixed }\alpha>0,
  \end{align*}
  the gap 
  converges to \(M_a/m_a\), viz.\  the gap in the error-residual relation
  \eqref{eq:err=ResA} for the state equation. Finally, let us consider
  the situation when the Tikhonov regularization $\alpha$ varies and
  suppose that we can construct $\tilde{x}$ such that 
  \begin{align}
    \label{Eq:Quotient=alpha}
    \frac{\delta^*_\alpha(R\tilde x,\tilde x)}
    {\|\Res(\tilde x)\|_*}
    \lesssim
    \alpha.
  \end{align}
  Then the gap is \(O(1)\), uniformly in $\alpha$.
\end{rem}

\begin{rem}[Compactness and lower bound for $d_\alpha$-error]
\label{R:compactness-and-low-bd}
The lower bound in Theorem~\ref{T:compact-case} does not exploit the compactness assumptions on control and observation. Note that it does not dependent on $\alpha$ and its form already coincides with the lower bound in \eqref{eq:err=ResA} for the state equation. These two properties suggest that it cannot be improved by involving terms like $\delta^*_\alpha$. Numerical results corroborate this suspicion; see Figure~\ref{F:rates-distributed-control-example}~(A).
\end{rem}

\subsection{Using compactness  of unconstrained control and observation}
\label{sec:apost-unconstraint}
%
In this subsection, we assume, in addition to the compactness of the control operator $C$ and the observation operator $I$, that there are no constraints for the control, i.e.\  
we have \(K=Q\) and therefore  \(\Pi_K=\operatorname{id}_Q\).

In view of these assumptions, the abstract optimal control problem \eqref{min} is quadratic and the reduced and rescaled optimality system \eqref{vred} is linear. Correspondingly, the form
\begin{equation*}
  \begin{aligned}
    b_\alpha(v,\vphi)
    =
    a(v_1,\vphi_2)
    +
    a(\vphi_1,v_2)
    +
    \tfrac{1}{\sqrt\alpha}\left(C^*v_2,C^*\vphi_2\right)_Q
    -
    \tfrac{1}{\sqrt\alpha}\left(I v_1,I\vphi_1\right)_W
  \end{aligned}
\end{equation*}
is bilinear and symmetric and
\begin{equation}
  \label{norm_alpha}
  \|v-w\|_\alpha
  :=
  d_\alpha(v,w)
  =
  \|v-w\| + \frac1{\sqrt{\alpha}} \frac{M}{M_a} |v-w|
\end{equation}
is a norm, where $\snorm{\cdot}$ is given in \eqref{enorm}. Hence, combing Theorem~\ref{T:bK} and the inf-sup duality \eqref{inf-sup-duality} for $b_\alpha$ instead of $\mathbf{a}$, we obtain
\begin{equation}
  \label{alternative-equiv}
  \frac{1}{M_a}\norm{\Res(R\tilde{x})}_{\alpha,*}
  \leq
  \| x - R\tilde{x} \|
  \leq
  \frac{\kappa}{m_a} \norm{\Res(R\tilde{x})}_{\alpha,*}
\end{equation}
with
\begin{equation*}
  \norm{\Res(R\tilde{x})}_{\alpha,*}
  \defas
  \sup_{\vphi \in V \setminus \{0\}}
   \frac{ \langle{\Res(R\tilde x)}, \vphi\rangle }{\|\vphi\|_\alpha}.
\end{equation*}
Notice that, in contrast to  \eqref{equivalence-for-x-Rtildex},  this equivalence involves like \eqref{eq:apostr-Recon} the norm $\|\cdot\|$, which does not depend on $\alpha$. This fact leads to the following alternative with $\norm{\cdot}$ as error norm of Theorem~\ref{T:compact-case}.

\begin{thm}[Bounding the $\norm{\cdot}$-error -- compact and unconstrained case]
\label{T:apostNCC}
Let  \(x=(u,z)\in V\) be the solution to~\eqref{vred} and $\tilde{u} \in V_2$ be some approximate state and $\tilde{z} \in V_1$ some approximate rescaled adjoint state. Writing \(\tilde x = (\tilde{u}, \tilde{z}) \), their error in the $\alpha$-independent norm $\norm{\cdot}$ is bounded from above and below as follows:
  \begin{align*}
    \|x-\tilde x\|
    \le
    \frac1{m_a} \left(
    1 +  \kappa\,\frac{M_a}{M} \,\frac{|\tilde x-R \tilde x| }{ \|\Res(\tilde x)\|_*}
    \right)
    \|\Res(\tilde x)\|_*
  \end{align*}
  and
  \begin{align*}
    \frac1{M_a}\max\left\{
    \frac{M_a\sqrt{\alpha}}{M_a\sqrt{\alpha}+M}, \,
    1 - \kappa\,\frac{M_a}{m_a}\frac{M_a}{M} \,
    \frac{|\tilde x-R\tilde x|}{\|\Res(\tilde x)\|_*}
    \right\} \;
    \|\Res(\tilde x)\|_*
    \le
    \|x-\tilde x\|.
  \end{align*}
\end{thm}

\begin{proof}
  To verify the upper bound, we start by  applying the triangle inequality
  \begin{align*}
    \|x-\tilde x\|
    \le
    \|x-R\tilde x\|
    +
    \|R\tilde x-\tilde x\|.
  \end{align*}
  For the second term on the right-hand side, we use \eqref{eq:apostr-Recon} as in the proof of Theorem~\ref{T:compact-case}. For the first term,  we employ \eqref{alternative-equiv} and \eqref{Eq:x-Rtildex} to obtain
  \begin{equation}
    \label{x-Rtildex-compact-unconstrainted-case}
    \begin{aligned}
      \frac{m_a}{\kappa}\, \|x-R\tilde x\|
      &\le
      \norm{ \Res(R\tilde x)}_{\alpha,*}
      =
      \sup_{\vphi \in V}
      \frac{c_\alpha(R\tilde{x}-\tilde{x},\vphi)}{\|\vphi\|_\alpha}
      \le
      \frac{M_a}{M}|\tilde x-R\tilde x|. 
    \end{aligned}
  \end{equation}
  This alternative to \eqref{Eq:APostReminder} concludes the proof of the upper bound.

  We turn to the lower bound. On the one hand, Theorem~\ref{T:Apost-1} from the direct approach and the definition \eqref{norm_alpha} of $\norm{\cdot}_\alpha$  imply
  \begin{align*}
    \frac{1}{M_a}\|\Res(\tilde x)\|_*\le \|x-\tilde x\|_\alpha \le
    \left(1+\frac{M}{M_a}\frac1{\sqrt\alpha}\right)\|x-\tilde x\|.
  \end{align*}
  This establishes the first option of the max.
  On the other hand, starting with \eqref{eq:apostr-Recon}, applying
  a triangle inequality and then \eqref{x-Rtildex-compact-unconstrainted-case}, we can
  deduce
  \begin{align*}
    \frac{1}{M_a}\|\Res(\tilde x)\|_*&\le \|R\tilde x-\tilde x\|\le
                                       \|x-\tilde x\|+\| x-R\tilde x\|
    \\
                                     &\le \|x-\tilde x\|+\frac{\kappa}{M}\frac{M_a}{m_a}|\tilde x-R\tilde x|.
  \end{align*}
  This shows the second option of the max and the proof is finished.
\end{proof}

\begin{rem}[Gap in bounding the $\norm{\cdot}$-error -- compact and unconstrained case]
\label{R:abstractCompact2}
%
If we do not exploit the compactness by inserting the  bound $\snorm{\tilde{x}-R\tilde{x}} \leq M \norm{\tilde{x}-R\tilde{x}}$ in the bounds of Theorem~\ref{T:apostNCC},  only the first option in the $\max$ of the lower bound applies and the gap is  $O(\kappa/\sqrt{\alpha}) = O(1/\alpha)$ as $\alpha \to 0$, the same order as in the corresponding case of Remark~\ref{R:abstractCompact1}. Recall however that Theorem~\ref{T:apostNCC} considers an error notion that is independent of $\alpha$. 

Similarly to Remark~\ref{R:abstractCompact1}, if the compactness assumptions lead to
\begin{align*}
    \frac{\snorm{\tilde{x} - R\tilde x}}
    {\|\tilde x-R\tilde x\|}
    \to
    0
    \qquad\text{for }\tilde  x\to x \text{ and fixed }\alpha>0,
\end{align*}
the gap between upper and lower bound converges to \(M_a/m_a\), viz.\  the gap in the error-residual relation \eqref{eq:err=ResA} for the state equation. Finally, if the Tikhonov regularization $\alpha$ varies and suppose that we can construct $\tilde{x}$ such that
  \begin{align}
    \label{Eq:Quotient=alphaNCC}
    \frac{\snorm{\tilde x  - R\tilde x}}{\|\Res(\tilde x)\|_*}
    \lesssim
    \sqrt{\alpha},
  \end{align}
  then the gap is \(O(1)\), uniformly in $\alpha$. This uses the
  second option in the $\max$ of the lower bound in
  Theorem~\ref{T:apostNCC} and condition \eqref{Eq:Quotient=alphaNCC}
  can be 
  less restrictive than its counterpart
  \eqref{Eq:Quotient=alpha} in Remark~\ref{R:abstractCompact1}. 
\end{rem}

\section{Applications to optimal control problems}
\label{sec:applications}
%
%
This section proceeds with the a~posteriori analysis for the abstract
optimization problem~\eqref{min}.  Recall that the remaining tasks
are to split the error bounds of
Section~\ref{sec:abstract-apost:res=err} into local contributions and,
then, to split the latter into computable and oscillatory parts. To
this end, one can apply
\begin{itemize}
\item classical techniques, see, e.g., \cite{AinsworthOden:2000,Verfuerth:2013}, 
leading to an equivalence of error and estimator up to so-called oscillation, or
\item more recent techniques, see
\cite[Sections~3 and 4]{KreuzerVeeser:2021}, \cite[Section 4]{BonCanNocVee:2024}, leading to a strict
equivalence.
\end{itemize}
In both cases, various estimator types can be chosen for the
computable parts. In any case, these tasks are specific to the
functional setting of the state equation, the control and observation
operators and the discretization. 
We therefore exemplify these tasks by considering finite element
discretizations of optimal control problems with distributed and
boundary control. Doing so, we focus on the possible interplay between
\(\|\Res(\tilde x)\|_*\) and \(\delta^*_\alpha(R\tilde x,\tilde
x)\),  \(\delta_\alpha(R\tilde x,\tilde x)\) or \(|R\tilde
x-\tilde x|\). 
Moreover, we numerically study the behavior of the derived
a~posteriori bounds. The adaptive simulations where carried
out in the DUNE framework \cite{Bastian_etal-DUNE:2021}.

In what follows, we shall use the following notation. For a Lebesgue measurable set
\(\omega\subset \R^d\),  we denote by \(L^2(\omega)\) the 
space of 
square integrable functions on \(\omega\) and define its norm by 
\(\|\cdot\|_{L^2(\omega)}^2:=\int_\omega |\cdot|^2\). The
space of functions having also weak first derivatives in
\(L^2(\omega)\) is denoted by \(H^1(\omega)\) with norm defined by
\(\|\cdot\|_{H^1(\omega)}^2\defas\|\nabla
\cdot\|_{L^2(\omega)}^2+\|\cdot\|_{L^2(\omega)}^2\)
and \(\mathring
H^1(\omega)\) is its closed subspace of functions with zero trace.
Recalling Poincar\'e's inequality \(\|v\|_{L^2(\omega)}\le C_P
\|\nabla v\|_{L^2(\omega)}\), \(v\in \mathring{H}^1(\omega)\), an alternative norm on
\(\mathring{H}^1(\omega)\) is given by \(\|\nabla\cdot\|_{L^2(\omega)}\).
The dual space of
\(\mathring H^1(\omega)\) is denoted by \(H^{-1}(\omega)\) and equipped with the
operator norm \(\|\cdot\|_{H^{-1}(\omega)}\). 
By virtue of the Riesz map, \(L^2(\omega)\) is identified with its
dual space.

\subsection{Distributed control}
\label{sec:Example}
In this section, we show how the a posteriori bounds from Section~\ref{sec:abstract-apost:res=err} are applied to optimal control problems with possibly constrained distributed control. 

Let \(\Omega\subset \R^d\), \(d>1\), be a domain with polyhedral
boundary \(\Gamma\). For a subdomain \(\Omega_{\mathsf{Q}}\subseteq \Omega\) and
some bounds \(a\in \R\cup\{-\infty\}\) and \(b\in \R\cup\{\infty\}\)
with \(a<b\), let
\begin{align}
\label{Eq:model-oc-dc;K}
  K
  =
  \{q\in L^2(\Omega_{\mathsf{Q}}) \mid
    a\le q\le b~\text{a.e.\ in}~\Omega_{\mathsf{Q}}\} 
\end{align}
be the set of admissible controls. For a target function \(u_d\in
L^2(\Omega_{\mathsf{W}})\) supported in the subdomain
\(\Omega_{\mathsf{W}}\subseteq 
\Omega\) and cost parameter \(\alpha>0\), we
consider
\begin{gather}\label{Eq:model-oc}
  \begin{gathered}
    \min_{ (q,u)\in K \times \mathring H^1(\Omega)}
    \frac12\|u-u_d\|^2_{L^2(\Omega_{\mathsf{W}})}
    +
    \frac{\alpha}2\|q\|_{L^2(\Omega_{\mathsf{Q}})}^2
    \\
    \text{subject to}\quad -\Delta u= f + q\chi_{\Omega_{\mathsf{Q}}}~\text{in}~\Omega
    \quad\text{and}\quad u=0~\text{on}~\Gamma.
  \end{gathered}
\end{gather}
Here, \(\chi_{\Omega_{\mathsf{Q}}}\) denotes the indicator function on
\(\Omega_{\mathsf{Q}}\) and \(q\chi_{\Omega_{\mathsf{Q}}}\) is
considered the zero extension of \(q\in L^2(\Omega_{\mathsf{Q}})\) to
\(\Omega\). This problem fits into the framework of
Section~\ref{sec:2} with the Hilbert spaces
\begin{alignat*}{2}
  V_1&=V_2= \mathring{H}^1(\Omega),
  &\quad(v_1,v_2)_{V_i}
  &=
  (\nabla v_1,\nabla v_2)_{L^2(\Omega)},
  \;\; i=1,2,
  \\
  Q&=L^2(\Omega_{\mathsf{Q}}),
  &(q_1,q_2)_Q &= (q_1,q_2)_{L^2(\Omega_{\mathsf{Q}})},
  \\
  W&=L^2(\Omega_{\mathsf{W}}), \quad
  &(w_1,w_2)_W &= (w_1,w_2)_{L^2(\Omega_{\mathsf{W}})}. 
\end{alignat*}
The remaining ingredients are given by~\eqref{Eq:model-oc-dc;K},
\begin{gather*}
  a(v,\vphi)
  =
  \int_\Omega \nabla v\cdot\nabla\vphi 
  =
  (\nabla v,\nabla \vphi)_{L^2(\Omega)},
  \qquad
  m_a = 1 = M_a,
  \\
  \langle Cq,\vphi \rangle
  =
  \int_{\Omega_{\mathsf{Q}}} q \vphi
  =
  (q,\vphi_{|\Omega_{\mathsf{Q}}} )_{L^2(\Omega_{\mathsf{Q}})}
  =
  ( q, C^* \vphi)_{L^2(\Omega_{\mathsf{Q}})},\qquad M=C_P,
  \\
  Iv = v_{|\Omega_{\mathsf{W}}},
  \qquad
  \langle I^*w, v \rangle
  =
  \int_{\Omega_{\mathsf{W}}} v w
  =
  (v_{|\Omega_{\mathsf{W}}},w)_{L^2(\Omega_\mathsf{W})},
  \\
  \Pi_K q=\min\left\{\max\{q,\,
    a\},\,b\right\}\quad\text{a.e. in}~\Omega_Q, 
\end{gather*}
where \(\langle \cdot,\cdot\rangle\) denotes the duality pairing
in \( H^{-1}(\Omega) \times \mathring{H}^1(\Omega)\).

The variational formulation of the reduced and rescaled optimality system~\eqref{mod-vred-rows} reads:
find \((u,z)\in \mathring{H}^1(\Omega)\times \mathring{H}^1(\Omega)\) such that
\begin{subequations}
  \label{mod-vred-distributed_control}
  \begin{align}
    &\forall \vphi_1\in \mathring{H}^1(\Omega)
    & \int_\Omega\nabla\vphi_1\cdot\nabla z - \tfrac{1}{\sqrt{\alpha}} \int_{\Omega_{\mathsf{W}}} u\vphi_1
    &=
      -\tfrac{1}{\sqrt{\alpha}} \int_{\Omega_{\mathsf{W}}} \ud\vphi_1,
    \\
    &\forall \vphi_2\in \mathring{H}^1(\Omega)
    &\int_\Omega\nabla u\cdot\nabla\vphi_2 - \int_{\Omega_{\mathsf{Q}}}\Pi_K(- \tfrac{1}{\sqrt{\alpha}} z)\vphi_2
    &= \dual{ f}{\vphi_2}
      .
  \end{align}
\end{subequations}

For its discretization, 
we use Lagrange finite elements. To this
end, let \(\mesh\) be a simplicial face-to-face (conforming) mesh 
of the domain \(\Omega\). Denoting by \(\vertices\) the vertices of
\(\mesh\), we define the star around a vertex \(z\in \vertices\) by
\begin{align*}
  \omega_z\defas\bigcup\{K\in\mesh\mid z\in K\}\quad\text{with
  diameter}\quad h_z=\operatorname{diam}(\omega_z).
\end{align*}
The discrete spaces associated with Lagrange finite elements of degree \(\ell>0\) are then given by
\begin{align*}
  S_\ell^1(\mesh)&\defas \left\{ v\in H^1(\Omega)\mid v_{|K}\in \mathbb{P}_\ell(K),
                   \forall K\in \mesh\right\}
                   \intertext{and}
                   \mathring{S}_\ell^1(\mesh)&\defas {S}_\ell^1(\mesh)\cap \mathring{H}^1(\Omega),
\end{align*}
where \(\mathbb{P}_\ell(K)\) denotes the set of polynomials up to degree
\(\ell\) on \(K\). The variational discretization
of~\eqref{mod-vred-distributed_control} then reads as follows:
find \((U,Z)\in
\mathring{S}_\ell^1(\mesh)\times \mathring{S}_\ell^1(\mesh)\) such that
\begin{subequations}
  \label{mod-vred-dc-discretisation}
  \begin{align}
    &\forall \Phi_1\in \mathring{S}_\ell^1(\mesh) 
    & \int_\Omega\nabla\Phi_1\cdot\nabla Z - \tfrac{1}{\sqrt{\alpha}} \int_{\Omega_{\mathsf{W}}} U\Phi_1
    &=
      -\tfrac{1}{\sqrt{\alpha}} \int_{\Omega_{\mathsf{W}}} \ud\Phi_1,
    \\
    &\forall \Phi_2\in \mathring{S}_\ell^1(\mesh) 
    &\int_\Omega \nabla U\cdot\nabla\Phi_2 - \int_{\Omega_{\mathsf{Q}}} \Pi_K(- \tfrac{1}{\sqrt{\alpha}} Z)\Phi_2
    &= \dual{f}{\Phi_2}
      .
  \end{align}
\end{subequations}
Consequently, in the solution \(X=(U,Z)\) of~\eqref{mod-vred-dc-discretisation}, the residual 
\begin{align}
\label{Eq:ResExample}
  \begin{aligned}
    \big\langle \Res(X),\, (\vphi_1,\vphi_2) \big\rangle
    &\defas 
    \dual{f}{\vphi_2}_{\mathring{H}^1(\Omega)} -\int_\Omega\nabla U\cdot\nabla\vphi_2 +
    \int_{\Omega_{\mathsf{Q}}} \Pi_K(- \tfrac{1}{\sqrt{\alpha}} Z)\vphi_2
    \\
    &\quad
    - \tfrac1{\sqrt{\alpha}} \int_{\Omega_{\mathsf{W}}} u_d\vphi_1
    -
    \int_\Omega\nabla\vphi_1\cdot\nabla Z
    +
    \tfrac{1}{\sqrt{\alpha}} \int_{\Omega_{\mathsf{W}}} U\vphi_1 ,
  \end{aligned}
\end{align}
for \((\vphi_1,\vphi_2)\in \mathring{H}^1(\Omega)\times \mathring{H}^1(\Omega)\),
satisfies the 
orthogonality condition
\begin{align}
\label{Eq:Galerkin-orthogonality}
  \langle\Res(X),\,\Phi\rangle=0\qquad\text{for
  all}~\Phi\in \mathring{S}_\ell^1(\mesh) \times \mathring{S}_\ell^1(\mesh) .
\end{align}

Using standard arguments, see, e.g.,\cite[Lemma 4]{KreuzerVeeser:2021}, we can split the residual norm into local contributions such that 
\begin{equation}
\label{Eq:Res-localisation}
\begin{aligned}
  \frac{1}{d+1}\sum_{z\in\vertices}\|\Res(X)\|^2_{H^{-1}(\omega_z)}
  &\le
  \|\Res(X)\|^2_{H^{-1}(\Omega)}
\\
  &\le
  C_{\mesh}\sum_{z\in\vertices}\|\Res(X)\|^2_{H^{-1}(\omega_z)},
\end{aligned}
\end{equation}
where the constant \(C_\mesh\) only depends on the shape regularity of
the mesh \(\mesh\). Notice that each contribution $\|\Res(X)\|^2_{H^{-1}(\omega_z)}$, $z \in \vertices$, is a  local quantity once the finite element solution $X=(U,Z)$ from \eqref{mod-vred-dc-discretisation} is available by  means of a global solve.
Combining this `localization' with  Theorem~\ref{T:Apost-1} of the direct approach readily provides the following a~posteriori bounds.

\begin{thm}[Bounding $d_\alpha$-error for distributed control -- general case]
\label{T:Apost1-ExampleCC}
Let \(x=(u,z)\) be the exact states of the optimal control problem
\eqref{Eq:model-oc}, where the adjoint state is rescaled, cf.~\eqref{mod-vred-distributed_control}. Furthermore, let \(X=(U,Z)\) be
their finite element approximations
from~\eqref{mod-vred-dc-discretisation}. Then, we have for the
residual  
 defined in~\eqref{Eq:ResExample} the equivalence 
  \begin{align*}
    \frac{1}{(d+1)}\sum_{z\in\vertices}\left\|\Res(\tilde x)\right\|_{H^{-1}(\omega_z)}^2
    \le
    d_\alpha(x, X)^2
    \le
    \kappa\, C_\mesh \sum_{z\in\vertices}\|\Res(X)\|^2_{H^{-1}(\omega_z)},
  \end{align*}
  where the constant \(C_\mesh\) depends on the shape regularity of
  the mesh and \(\kappa\) is defined in~\eqref{kappa}.
\end{thm}

Thanks to the  compact embedding \(\mathring{H}^1(\Omega)\subset L^2(\Omega)\),
the operators \(C^*\) and \(I\) are compact.  This allows applying the results of the compact approach in Section~\ref{sec:apost-ellRecon}. To this end, we use the Lipschitz continuity of \(\Pi_K\) with Lipschitz constant $1$ to deduce that
\begin{align}\label{Eq:delta<L2norm}
  \begin{aligned}
    \delta_\alpha^*(R X,X)^2\le\delta_\alpha(R X,X)^2
    &\le \norm{R_2X-Z }_{L^2(\Omega)}^2 +
    \norm{R_1X-U}_{L^2(\Omega)}^2,
  \end{aligned}
\end{align}
where \(R=(R_1,R_2): \mathring{S}_\ell^1(\mesh) \times \mathring{S}_\ell^1(\mesh) \to
\mathring{H}^1(\Omega)\times \mathring{H}^1(\Omega)\) is the auxiliary operator defined in~\eqref{eq:Rell}.
Combining the definitions of the auxiliary operator $R$ and the finite element solution $X$, we find the orthogonality relationships
\begin{align}\label{Eq:Ritz}
  \int_{\Omega}\nabla (R_1X-U) \cdot \nabla\Phi_2
  =
  0
  =
  \int_{\Omega}\nabla (R_2X-U)\cdot\nabla \Phi_1\quad\forall \Phi_1,\Phi_2\in \mathring{S}_\ell^1(\mesh)
\end{align}
In other words, $U$ and $Z$ are, respectively, the Ritz projections in $\mathring{S}_\ell^1(\mesh)$
of $R_1 X$ and $R_2 X$ with respect to the bilinear form \((\nabla\cdot,\nabla\cdot)_{L^2(\Omega)}\). Taking into account also the orthogonality \eqref{Eq:Galerkin-orthogonality} of $\Res(X)$, we can thus use a well-known argument, see, e.g., \cite[Section~2.4]{AinsworthOden:2000}, based upon duality and regularity, to bound the \(L^2\)-errors in~\eqref{Eq:delta<L2norm} in terms of the residual $\Res(X)$.  For simplicity, we shall assume that the domain \(\Omega\) is convex and resort to the following well-known \(H^2\)-regularity result for the Poisson problem; compare with \cite[(3,1,2,2) and Lemma 3.2.1.2]{Grisvard:2011}.

\begin{prop}[Extra regularity for distributed control]
\label{P:regularity-PoissonL2}
Let \(\Omega\subset \R^d\) be a convex domain. For any source \( g \in L^2(\Omega)\), the unique solution \( v_g\in \mathring{H}^1(\Omega)\) of the Poisson problem
  \begin{align*}
    \forall v\in \mathring{H}^1(\Omega)
   \qquad
   \int_\Omega \nabla  v_g\cdot\nabla v
   =
   \int_\Omega gv
  \end{align*}
  satisfies
  \begin{align*}
    v_g\in H^2(\Omega)\qquad\text{and}\qquad |v_g|_{H^2(\Omega)}
    \le 
    \|g\|_{L^2(\Omega)},
  \end{align*}
  where 
  \(|\cdot|_{H^2(\Omega)}\) denotes the \(H^2(\Omega)\)-seminorm.
\end{prop}

Using Proposition~\ref{P:regularity-PoissonL2} in the cited duality
argument then leads to the following a posteriori upper bound.
\begin{lem}[Upper bound for compact error -- distributed control]
\label{L:apost-deltaK}
Let \(\Omega\subset \R^d\) be a convex polyhedral domain. The $L^2$-errors in \eqref{Eq:delta<L2norm} are bounded in terms of the residual of $X=(U,Z)$:
  \begin{align*}
    \|R_2X-Z\|_{L^2(\Omega)}^2+\|R_1X-U\|_{L^2(\Omega)}^2
    \le 
    C_\mesh^2\sum_{z\in\vertices}h_z^2 \|\Res(X)\|^2_{H^{-1}(\omega_z)},
  \end{align*}
  where \(C_\mesh\) is the constant
  from~\eqref{Eq:Res-localisation}.
\end{lem}

\begin{proof}
We sketch the proof only for  the second term \( \|R_1X-U\|_{L^2(\Omega)}^2\); the same argument applies also to the first term. According to Proposition~\ref{P:regularity-PoissonL2}, there is  \(\psi\in H^2(\Omega)\cap \mathring{H}^1(\Omega)\) with
\begin{align}\label{Eq:regRes2}
    -\Delta \psi=R_1X-U\quad\text{in}~\Omega\qquad \text{and}\qquad
    |\psi|_{H^2(\Omega)}
    \le 
    \|R_1X-U\|_{L^2(\Omega)}.
\end{align}
We denote by
\(\mathcal{I}_{\textsf{sz}}:\mathring{H}^1(\Omega)\to
  \mathring{S}_\ell^1(\mesh)\)
the Scott-Zhang quasi-interpolation operator~\cite{ScottZhang:1990}. Thanks to the definition of $R$ and the orthogonality~\eqref{Eq:Galerkin-orthogonality} of the residual, we deduce
\begin{align*}
    \|R_1X-U\|_{L^2(\Omega)}^2&=\int_\Omega \nabla (R_1X-U)\cdot\nabla
                                \psi
                                =
                                \langle \Res_2(X),\psi\rangle
    \\
                              &=\langle
                                \Res_2(X),\psi-\mathcal{I}_{\textsf{sz}}
                                \psi\rangle
                                =\sum_{z\in\vertices}\langle
                                \Res_2(X),(\psi-\mathcal{I}_{\textsf{sz}}
                                \psi)\phi_z\rangle
    \\
                              &\le
                                \sum_{z\in\vertices}\|
                                \Res_2(X)\|_{H^{-1}(\omega_z)}\|\nabla((\psi-\mathcal{I}_{\textsf{sz}}
                                \psi)\phi_z)\|_{L^2(\omega_z)},
\end{align*}
where we have used for the last equality that the Lagrange basis functions \(\phi_z\), \(z\in\vertices\), of \(S_1^1(\mesh)\) form a partition of unity and that \(\operatorname{supp}{\phi_z}=\omega_z\), \(z\in\vertices\). 
In view of \(\|\phi_z\|_{L^\infty(\omega_z)}=1\)  and \(\|\phi_z\|_{L^\infty(\omega_z)}\le C_\mesh h_z^{-1}\), standard interpolation estimates imply
\begin{equation}
\label{Eq:countering-cutoff-H2}
   \begin{aligned}
      \|\nabla((\psi &- \mathcal{I}_{\textsf{sz}} \psi)\phi_z)\|_{L^2(\omega_z)}
      \\
      &\le
      \|\nabla(\psi-\mathcal{I}_{\textsf{sz}} \psi)\|_{L^2(\omega_z)}
      +
      \|\nabla \phi_z\|_{L^\infty(\omega_z)} \|\psi-\mathcal{I}_{\textsf{sz}} \psi\|_{L^2(\omega_z)}
      \\
      &\le C_\mesh h_z |\psi|_{H^2(\tilde \omega_z)}.
    \end{aligned}
  \end{equation}
  Note that the constant \(C_\mesh\) may vary from occurence to
  occurence but each time
  only depends on the shape regularity of \(\mesh\). Here the domains
  \(\tilde \omega_z=\bigcup_{y\in \vertices\cap \omega_z}\omega_y\)
  are neighbourhoods of \(\omega_z\) that only overlap finitely often
  depending on the regularity of \(\mesh\). This together
  with~\eqref{Eq:regRes2} implies
  \begin{equation*}
    \|R_1X-U\|_{L^2(\Omega)}^2
    \le C_\mesh 
    \left( \sum_{z\in\vertices}h_z^2\|
    \Res_2(X)\|_{H^{-1}(\omega_z)}^2\right)^{\frac12} \|R_1X-U\|_{L^2(\Omega)}.
   \qedhere
  \end{equation*}
\end{proof}

Inserting Lemma~\ref{L:apost-deltaK} as well as the localization \eqref{Eq:Res-localisation} into Theorem~\ref{T:compact-case}, we obtain the
following alternative to Theorem~\ref{T:Apost1-ExampleCC}. 
\begin{thm}[Distributed control -- compact case]
\label{T:apost-ExampleCC}
 Suppose in addition to the setting of Theorem~\ref{T:Apost1-ExampleCC} that
  \(\Omega\subset \R^d\) is convex. Then we have
  \begin{multline*}
  	 \frac{1}{\sqrt{d+1}} \left(
  	  \sum_{z\in\vertices}\|\Res(X)\|_{H^{-1}(\omega_z)}^2
  	 \right)^{\frac{1}{2}}
  	 \le
  	 d_\alpha(x,X)
\\
    \le
    C
     \left(
      1+\frac{\kappa}{\sqrt{\alpha}}
       \,\left(
         \frac{\sum_{z\in\vertices}h_z^2\|\Res(X)\|_{H^{-1}(\omega_z)}^2}
           {\sum_{z\in\vertices}\|\Res(X)\|_{H^{-1}(\omega_z)}^2}
         \right)^{\!\frac12}\right)
    \left(
     \sum_{z\in\vertices}\|\Res(X)\|_{H^{-1}(\omega_z)}^2
    \right)^{\!\frac12}.
  \end{multline*}
The constant  \(C\) depends on the Poincar\'e 
  constant \(C_P\) and 
  the shape regularity of the mesh \(\mesh\);
  \(\kappa\) is  defined in~\eqref{kappa}.
\end{thm}

\begin{rem}[Limitations in exploiting compactness]
\label{R:limitations}	
In Theorem~\ref{T:apost-ExampleCC} we have exploited the compactness
of $C^*$ and $I$ to obtain the accelerating factors $h_z$ in front of
the local contributions $\| \Res(X) \|_{H^{-1}(\omega_z)}$, $z \in
\vertices$. Notice that the use of \eqref{Eq:delta<L2norm} entails
that Theorem~\ref{T:apost-ExampleCC} does not exploit the compactness
of $C$ and $I^*$ associated with the embedding $L^2(\Omega) \subset
H^{-1}(\Omega)$. Thus, the question arises  whether the upper bound in
Theorem~\ref{T:apost-ExampleCC} can be improved by directly bounding
the potentially smaller quantity \(\delta_\alpha^*(RX,X)\). The
line of argument allows for such an improvement in principle, but
hinges on the combination of regularity properties for the state
equation and its adjoint as well as on the order $\ell$ of their
finite element solutions. The former obstructs an improvement of
Theorem~\ref{T:apost-ExampleCC} in the case at hand. 

To illustrate this, let us consider the case with \(\Omega=\Omega_{\mathsf{Q}}=\Omega_{\mathsf{W}}\)
and \(a=-\infty\) and \(b=\infty\), i.e., \(\Pi_K=\operatorname{id}\), leading to
\begin{align*}
	{M^2}\delta_\alpha^*(RX,R)^2= \norm{R_2X-Z }_{H^{-1}(\Omega)}^2
	+
	\norm{R_1X-U}_{H^{-1}(\Omega)}^2.
\end{align*}
As in the proof of Lemma~\ref{L:apost-deltaK}, let us focus on the
second term on the right-hand side. In view of 
\begin{equation*}
 \norm{R_1X-U}_{H^{-1}(\Omega)}
 =
 \sup_{\vphi \in \mathring{H}^1(\Omega)}
  \frac{\langle R_1 X - U, \vphi \rangle}{\norm{\nabla\vphi}_{L^2(\Omega)}},
\end{equation*}
we consider
\begin{equation*}
  - \Delta \psi = \vphi \in \mathring{H}^1(\Omega) \text{ in }\Omega,
 \qquad
   \psi = 0 \text{ on } \partial\Omega.
\end{equation*}
If we had $\psi \in H^3(\Omega)$ with $\snorm{\psi}_{H^3(\Omega)} \leq C \| \nabla \vphi \|_{L^2(\Omega)}$ and $\ell > 1$, then minor modifications in the proof of  Lemma~\ref{L:apost-deltaK} would imply
 \begin{equation*}
	\|R_1X-U\|_{H^{-1}(\Omega)}^2
	\le 
	C_\mesh^2\sum_{z\in\vertices}h_z^4 \|\Res(X)\|^2_{H^{-1}(\omega_z)}.
\end{equation*}
However, the supposed regularity theorem is not true for polyhedral
domains or would not be useful in the case $\ell=1$ of linear finite
elements. As an alternative, one could invoke also more sophisticated
regularity theorems with weights. We do not consider this option here
for simplicity. 
\end{rem}

For the numerical comparison of bounds as in Theorem~\ref{T:Apost1-ExampleCC} and Theorem~\ref{T:apost-ExampleCC}, we consider
\begin{equation}
\label{distributed-control-example}
\begin{multlined}
	\min_{(q,u)\in K \times H^1( \Omega)}
	 \frac12 \norm{ u-\ud }_{L^2(\Omega_\textsf{W})}^2
	+
	\int_{ \Omega} g_1 u
	+
	\frac\alpha2 \norm{q}_{L^2(\Omega_\textsf{Q})}^2
\\
	\text{subject to}\quad
	-\Delta u = f+q \text{ in } \Omega\quad\text{and}\quad u=g_2 \text{ on } \partial \Omega,
\end{multlined}
\end{equation}
where the domains $\Omega$,  \(\Omega_\textsf{W}\), \(\Omega_\textsf{Q}\) as in Figure~\ref{F:distributed-control-domain}, 
$K = \{ q \in L^2(\Omega) \mid -1\leq q \leq 1 \}$,
$u= 3r^{\frac{4}{3}}\sin(\frac{4}{3} \theta) $,
$z=4(y-y^2)(1-x)(x+y)$,
$q = \chi_{\Omega_\textsf{Q}}\Pi_{[-1,1]} (\frac{-1}{\sqrt{\alpha}} z)$,
$f= -q$,
$\ud = u+\Delta z$,
$g_2=u$,  and
$g_1 = \chi_{\Omega_\textsf{Q}} \sqrt{\alpha} z$.

 \begin{figure}[ht]
 		\begin{tikzpicture}[scale=2.8]
 			\draw[axis] (-1,0) -- (1.2,0) node[right=2* \nudge cm] {\(x_1\)};
 			\draw[axis] (0,0) -- (0,1.2) node[above=2*\nudge cm] {\(x_2\)};
 			\begin{scope}
 				\draw[line] (0,0) -- (1,1) -- (-1,1) -- (0,0);
 				\clip (0,0) -- (1,1) -- (-1,1) -- (0,0);
 				\fill[pattern=north west lines] (-4,-4) rectangle (4,4);
 			\end{scope}
 			\begin{scope}
 				\draw[line] (0,0) -- (1,0) -- (1,1) -- (0,0);
 				\clip (0,0) -- (1,0) -- (1,1) -- (0,0);
 				\fill[pattern=north east lines] (-4,-4) rectangle (4,4);
 			\end{scope}
 			\draw[->] (1.25,0.5) -- (0.8,0.35);
 			\node at (1.35,0.5) {$\Omega_\textsf{Q}$};
 			\draw[->] (-0.5,0.35) -- (-0.2,0.5);
 			\node at (-0.55,0.3) {$\Omega_\textsf{W}$};
 			\node at (1,-0.1) {(1,\,0)};
 			\node at (1, 1.1) {(1,\,1)};
 			\node at (-1, 1.1) {(-1,\,1)};
 			\node at (-0.15,-0.1) {(0,\,0)};
 		\end{tikzpicture}
 		\caption{Domain $\Omega$ and subdomains
                 \(\Omega_\textsf{W}, \Omega_\textsf{Q}\) 
                  for Example~\eqref{distributed-control-example}.\label{F:distributed-control-domain}}
 	\end{figure}
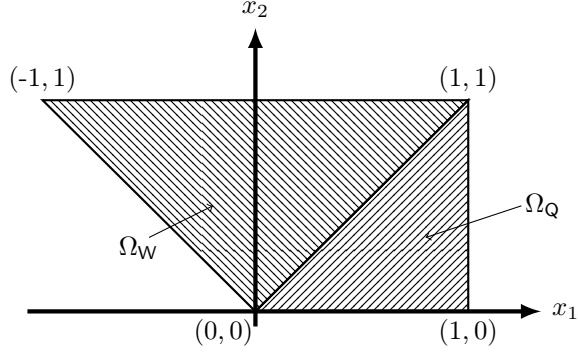

The numerical simulations are carried out with linear elements. Adaptive mesh refinement is driven by a standard residual error estimator, see, e.g.,\cite{Verfuerth:2013}, that quantifies the local residual norms in Theorem~\ref{T:Apost1-ExampleCC}. The estimator is scaled such that it coincides with the error for large \(\alpha=10^6\) and a fine, adaptive
grid, providing a benchmark close to the situation of a pure Poisson problem. To mark elements for refinement, D\"orfler's strategy~\cite{Doerfler:1996} is used with parameter \(0.6\). 

Figure~\ref{F:rates-distributed-control-example} displays the $d_\alpha$-error and the bounds in Theorem~\ref{T:Apost1-ExampleCC} and Theorem~\ref{T:apost-ExampleCC} using compactness, while Figure \ref{F:meshes--4p1} gives an idea of the underlying adaptive mesh refinement. First, let us observe that the error may or may not be close to the $\alpha$-independent lower bound on coarse meshes. Next, in line with the Remarks~\ref{R:vanish-regularization} and~\ref{R:abstractCompact1}, we see that the upper bound using compactness is worse than the general one from Theorem~\ref{T:Apost1-ExampleCC} on coarse meshes, but can provide a much smaller gap on fine meshes; see part (A) with $\alpha = 10^{-4}$. However, this improvement hinges on the relationship of $\alpha$ and the available computational resources; see part (B) with $\alpha=10^{-8}$. 

\begin{figure}[htb]
  \begin{subfigure}[b]{0.49\textwidth}
		\begin{tikzpicture}[scale=0.7]
			\begin{loglogaxis}[xlabel={Dofs}, ylabel={Size},legend pos=outer north east]
				\addplot[black, mark=+] table[x=Dofs , y= dkalpha] {./distributed-control-4p1.txt};
				\addlegendentry{$d_\alpha$ error}
				\addplot[red, mark=square*] table[x=Dofs , y= Asbou] {./distributed-control-4p1.txt};
				\addlegendentry{as. upper }
				\addplot[blue, mark=*] table[x=Dofs , y= Nasbou] {./distributed-control-4p1.txt};
				\addlegendentry{dir. upper }
				\addplot[blue, mark=o] table[ x=Dofs, y=esti ] {./distributed-control-4p1.txt};
				\addlegendentry{lower}
                                \legend{}
			\end{loglogaxis}
		\end{tikzpicture}
		\caption{ $\alpha = 10^{-4}$ }
	\end{subfigure}\hfill
	\begin{subfigure}[b]{0.49\textwidth}
			\begin{tikzpicture}[scale=0.7]
			\begin{loglogaxis}[xlabel={Dofs}, ylabel={Size}
                          ]
				\addplot[black, mark=+] table[x=Dofs , y= dkalpha] {./distributed-control-8p1.txt};
				\addlegendentry{$d_\alpha$-error}
				\addplot[red, mark=square*] table[x=Dofs , y= Asbou] {./distributed-control-8p1.txt};
				\addlegendentry{cubd}
				\addplot[blue, mark=*] table[x=Dofs , y= Nasbou] {./distributed-control-8p1.txt};
				\addlegendentry{gubd}
				\addplot[blue, mark=o] table[ x=Dofs, y=esti ] {./distributed-control-8p1.txt};
				\addlegendentry{glbd}
                                \legend{}
			\end{loglogaxis}
		\end{tikzpicture}
		\caption{ $\alpha = 10^{-8}$ }
	\end{subfigure}
	\caption{$d_\alpha$-error (+) and associated general upper (\textcolor{blue}{$\bullet$}) and lower (\textcolor{blue}{$\circ$}) bound, as well as upper bound with compactness (\textcolor{red}{$\textcolor{red}{\rule{.9ex}{.9ex}}$}) versus DOFs for Example~\eqref{distributed-control-example}.\label{F:rates-distributed-control-example}}
\end{figure}
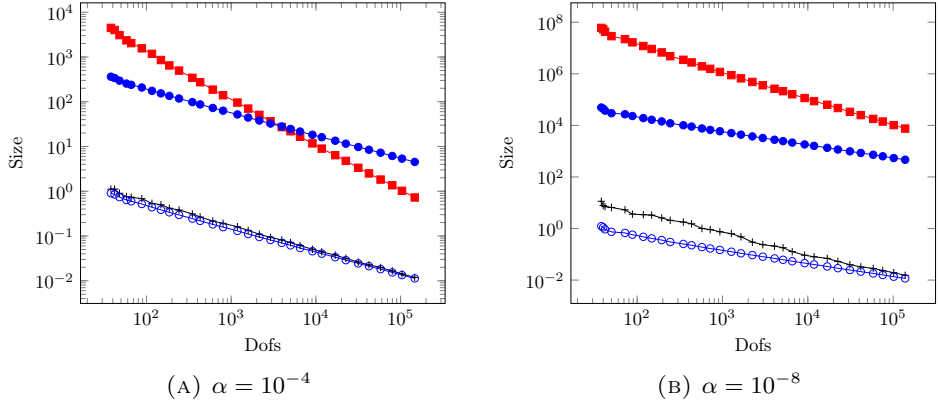

\begin{figure}[htb]
                      \includegraphics[width=.48\linewidth]{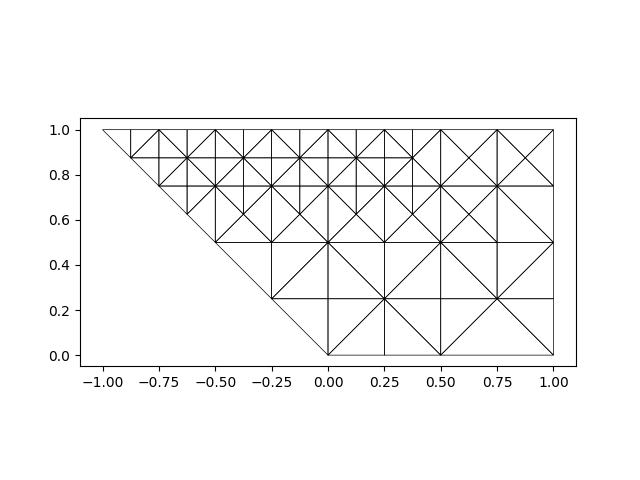}
                      \includegraphics[width=.48\linewidth]{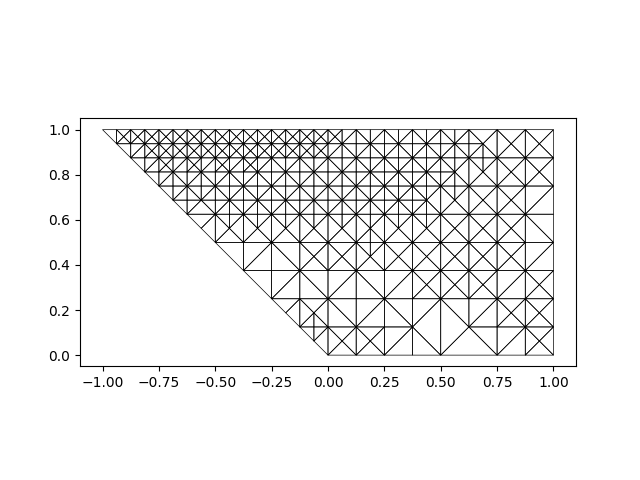}
                      \\
                      \includegraphics[width=.48\linewidth]{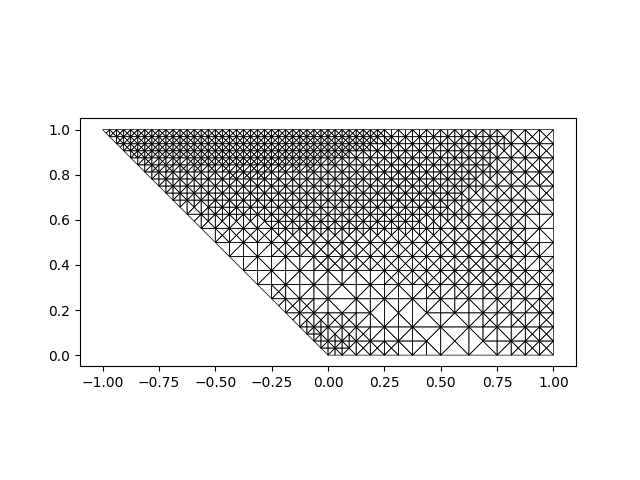}
                      \includegraphics[width=.48\linewidth]{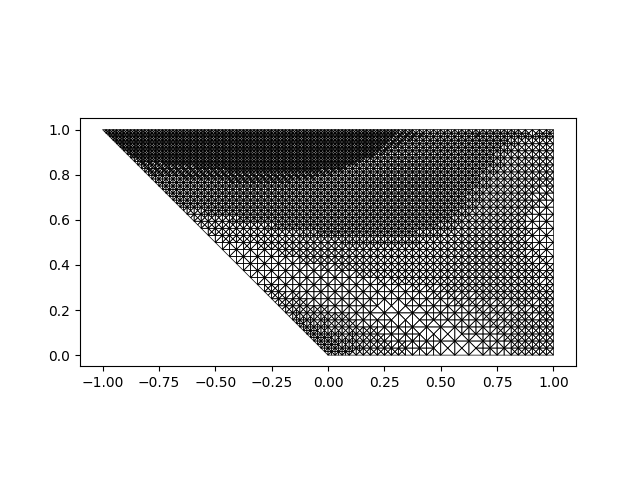}
                      \caption{Adaptive mesh refinement history for Example~\eqref{distributed-control-example} with $\alpha = 10^{-4}$.\label{F:meshes--4p1}}
\end{figure}

We turn to applying the results of Section~\ref{sec:apost-unconstraint} and suppose that there are no control constraints, i.e.\  \(a=-\infty\) and \(b=\infty\). Theorem~\ref{T:apostNCC}, combined with Lemma~\ref{L:apost-deltaK} and the localization~\eqref{Eq:Res-localisation}, immediately yields the following a posteriori bounds for the combined $\mathring{H^{1}}(\Omega)$-error of the states.
\begin{thm}[Bounding the $\norm{\cdot}$-error for distributed control]
\label{T:apost-ExampleNCC}
Suppose in addition to the setting in Theorem~\ref{T:apost-ExampleCC} that no control constraints apply.
Then we have 
 \begin{multline*}
    \max\left\{
     \frac{\sqrt{\alpha}}{\sqrt{\alpha}+C_P},
     1-C \kappa\, F\big( \Res(X) \big) 
    \right\}
 \left(\frac1{d+1}\sum_{z\in\vertices}\|\Res(X)\|_{H^{-1}(\omega_z)}^2\right)^{\frac12}
\\
 \le
 \|\nabla(x-X)\|_{L^2(\Omega)}
 \le
 C
    \left(1+\kappa\,F\big( \Res(X) \big) 
      \right)\left(\sum_{z\in\vertices}\|\Res(X)\|_{H^{-1}(\omega_z)}^2\right)^{\frac12} 
\end{multline*}
with
\begin{equation*}
 F\big( \Res(X) \big)
 =
 \left(\frac{\sum_{z\in\vertices}h_z^2\|\Res(X)\|_{H^{-1}(\omega_z)}^2}{\sum_{z\in\vertices}\|\Res(X)\|_{H^{-1}(\omega_z)}^2}\right)^{\frac12}.
\end{equation*}
The constant  \(C\) depends on the Poincar\'e 
constant \(C_P\) and 
the shape regularity of the mesh \(\mesh\);
\(\kappa\) is  defined in~\eqref{kappa}.
\end{thm}


%
%
\subsection{Boundary control}
\label{sec:bndry-control}
%
This section illustrates how to apply the a posteriori bounds from Section~\ref{sec:apost-ellRecon} to optimal control problems with constrained Neumann boundary control.

Let \(\Omega\subset \R^d\) be a domain with polyhedral boundary \(\Gamma\)
and outward-pointing normal $\mathsf{n}$. Given some lower bound \(a\in \R\cup\{-\infty\}\), let 
\begin{equation}
\label{Eq:model-oc-bc;K}
	K
	=
	\left\{ q\in L^2(\Gamma) \mid \int_\Gamma q\ge a \right\},
\end{equation}
denote the set of admissible controls. For targets \(u_d^\Omega\in L^2(\Omega)\), \(u_d^\Gamma\in L^2(\Gamma)\) and cost parameter \(\alpha>0\), we consider 
\begin{gather}
\label{Eq:model-oc-bc}
  \begin{multlined}
  \min_{(q,u)\in K \times H^1(\Omega)}
   \frac12\|u-u_d^\Omega\|^2_{L^2(\Omega)}
   +
   \frac12\|u-u_d^\Gamma\|_{L^2(\Gamma)}^2
   +
   \frac{\alpha}2\|q\|_{L^2(\Gamma)}^2
\\
  \text{subject to} \quad
  -\Delta u+u = 0~\text{in}~\Omega
  \quad\text{and}\quad
  \partial_{\textsf{n}}u=q~\text{on}~\Gamma,
  \end{multlined}
\end{gather}
where \(\partial_{\textsf{n}}\) denotes the normal derivative. 
%
This problem fits into the framework of Section~\ref{sec:2} with the Hilbert spaces
\begin{alignat*}{2}
 V_1&=V_2= H^1(\Omega),
 &(v_1,v_2)_{V_i}
 &=
 (v_1,v_2)_{H^1(\Omega)},
 \;\; i=1,2,
\\
 Q&=L^2(\Gamma),
 &(q_1,q_2)_Q &= (q_1,q_2)_{L^2(\Gamma)},
\\
 W&=L^2(\Omega)\times L^2(\Gamma), \quad
 &(w_1,w_2)_W &= (w_1^\Omega,w_2^\Omega)_{L^2(\Omega)}
 + (w_1^\Gamma, w_2^\Gamma)_{L^2(\Gamma)},
\end{alignat*}
writing $w_i = (w_i^\Omega, w_i^\Gamma) \in W$, $i=1,2$. The other ingredients  are given by \eqref{Eq:model-oc-bc;K},
\begin{gather*}
  a(v,\vphi)
  =
  \int_\Omega \nabla v\cdot\nabla\vphi + v\vphi
  =
  (v,\vphi)_{H^1(\Omega)},
  \qquad
  m_a = 1 = M_a,
  \\
  \langle Cq,\vphi \rangle
  =
  \int_\Gamma q \vphi
  =
  (q,\vphi_{|\Gamma})_{L^2(\Gamma)}
  =
  ( q, C^* \vphi)_{L^2(\Gamma)}, \qquad M_C=C_\Gamma,
  \\
  Iv = (v,v_{|\Gamma}),
  \qquad
  \langle I^*(w^\Omega,w^\Gamma), v \rangle
  =
  \int_\Omega v w^\Omega + \int_\Gamma v_{|\Gamma} w^\Gamma, \qquad M_I=(1+C_\Gamma),
  \\
  \Pi_K v=v+\frac1{|\Gamma|}\max\left\{0,a-\int_\Gamma v\right\},
\end{gather*}
where \(C_\Gamma\) is the embedding constant \(H^1(\Omega)\subset
L^2(\Gamma)\) and we write $\langle \cdot,\cdot\rangle$ for the
duality pairing in $H^1(\Omega)^* \times H^1(\Omega)$.

The variational formulation of the reduced and rescaled optimality system~\eqref{mod-vred-rows} reads:
find \((u,z)\in H^1(\Omega)\times H^1(\Omega)\) such that, for all $\vphi_1, \vphi_2\in H^1(\Omega)$, 
\begin{subequations}
\label{mod-vred-bndry_control}
\begin{gather}
 \begin{multlined}
 	\int_\Omega \nabla\vphi_1\cdot\nabla z +\vphi_1 z
    -
 	\tfrac{1}{\sqrt{\alpha}} \left(
 	\int_\Omega u\vphi_1 + \int_{\Gamma} u\vphi_1
 	\right)
\\ 
 	=
 	-\tfrac{1}{\sqrt{\alpha}} \left(
 	\int_\Omega \ud^\Omega\vphi_1
 	+
 	\int_\Gamma \ud^\Gamma\vphi_1
 	\right),
 \end{multlined}
\\
 \int_\Omega\nabla u\cdot\nabla\vphi_2+u\vphi_2
 -
 \int_\Gamma\Pi_K(- \tfrac{1}{\sqrt{\alpha}} z)\vphi_2
 =
 0.
\end{gather}
\end{subequations}

Using the finite element framework of Section~\ref{sec:Example}, its discretisation reads as follows:
find \((U,Z)\in S_\ell^1(\mesh)\times S_\ell^1(\mesh)\) such that, for all $\Phi_1, \Phi_2 \in   S_\ell^1(\mesh)$, 
\begin{subequations}
\label{mod-vred-bc-discretisation}
\begin{gather}
 \begin{multlined}
 \int_\Omega \nabla\Phi_1\cdot\nabla Z +\Phi_1 Z
  - \tfrac{1}{\sqrt{\alpha}} \left( 
   \int_\Omega U\Phi_1 + \int_{\Gamma} U\Phi_1
  \right)
\\
  =
  -\tfrac{1}{\sqrt{\alpha}} \left(
   \int_\Omega \ud^\Omega\Phi_1 + \int_\Gamma \ud^\Gamma\Phi_1
  \right),
 \end{multlined}
\\
 \int_\Omega\nabla U\cdot\nabla\vphi_2+U\vphi_2
 -
 \int_\Gamma\Pi_K(- \tfrac{1}{\sqrt{\alpha}} Z)\vphi_2
 = 0.
\end{gather}
\end{subequations}
Consequently, in the solution \(X=(U,Z)\) of~\eqref{mod-vred-bc-discretisation}, the residual 
\begin{align}\label{Eq:ResExample-bc}
  \begin{aligned}
    \big\langle \Res(X),\, (\vphi_1,\vphi_2) \big\rangle
    &\defas
    -\int_\Omega\nabla U\cdot\nabla\vphi_2 -U\vphi_2+
    \int_\Gamma\Pi_K(- \tfrac{1}{\sqrt{\alpha}} Z)\vphi_2
    \\
    &\quad
    -\tfrac1{\sqrt{\alpha}} \int_\Omega u_d^\Omega\vphi_1
    -
    \tfrac1{\sqrt{\alpha}} \int_\Gamma u_d^\Gamma\vphi_1
    \\
    &\quad
    -\int_\Omega\nabla\vphi_1\cdot\nabla Z -\vphi_1Z
    +
    \tfrac{1}{\sqrt{\alpha}} \int_\Omega U\vphi_1
    +
    \tfrac{1}{\sqrt{\alpha}} \int_\Gamma U\vphi_1 ,
  \end{aligned}
\end{align}
 \((\vphi_1,\vphi_2)\in H^1(\Omega)\times H^1(\Omega)\),
satisfies the orthogonality condition
\begin{align}\label{Eq:Galerkin-orthogonality-bc}
  \langle\Res(X),\,\Phi\rangle=0\qquad\text{for
  all}~\Phi\in {S}_\ell^1(\mesh) \times {S}_\ell^1(\mesh) .
\end{align}
Similarly as in~\eqref{Eq:Res-localisation}, we can localize the norm of the residual by 
\begin{align}\label{Eq:Res-localisation-bc}
  \frac{1}{d+1}\sum_{z\in\vertices}\|\Res(X)\|^2_{H_z^*}\le
  \|\Res(X)\|^2_{(H^{1}(\Omega))^*}\le C_{\mesh}\sum_{z\in\vertices}\|\Res(X)\|^2_{H_z^*}
\end{align}
with \(H_z := H^{-1}(\omega_z)\) for interior vertices  \(z\in
\vertices\cap\Omega\) and 
\(\{ v \in H^1(\omega_z) \mid v = 0 \text{ on } \partial\omega_z
\setminus \partial\Omega\}\) for $z \in \vertices\cap\Gamma$.
For the proof, we refer to Lemma~\ref{L:delta<hRes;Besov}, 
where similar
arguments are used. 
Inserting the localization \eqref{Eq:Res-localisation-bc} into Theorem~\ref{T:Apost-1} yields the following result.

\begin{thm}[Bounding $d_\alpha$-error for boundary control -- general case]
\label{T:Apost1-ExampleBC}
Let \(x=(u,z)\) be the exact states of the optimal control problem~\eqref{Eq:model-oc-bc}, where the adjoint state is rescaled; cf.~\eqref{mod-vred-bndry_control}. Furthermore, let \(X=(U,Z)\) be their finite element approximations from~\eqref{mod-vred-bc-discretisation} and define its residual by~\eqref{Eq:ResExample-bc}. Then we have the equivalence
\begin{align*}
    \frac{1}{d+1}\sum_{z\in\vertices}\|\Res(X)\|^2_{H_z^*}
    \le
    d_\alpha(x,X)^2
    \le
    \kappa \,C_\mesh \sum_{z\in\vertices}\|\Res(X)\|^2_{H_z^*},
\end{align*}
where the constant \(C_\mesh\) depends on the shape regularity of
  the mesh and \(\kappa\) is defined in~\eqref{kappa}.
\end{thm}
Next, we shall use the fact, that the operators $C^*$ and $I$ involved
in the definition \eqref{eq:Rell} of the reconstruction 
are compact. Indeed,  the compactness of $C^*$ originates in the
trace evaluation $H^1(\Omega) \ni \vphi \mapsto \vphi_\Gamma \in
L^2(\Gamma)$, while the one of the observation operator $I$ arises
also with the help of the embedding $H^1(\Omega) \subset
L^2(\Omega)$. We therefore can apply the results of
Section~\ref{sec:apost-ellRecon} and need to quantify the compactness. To prepare this, we
use the Lipschitz continuity of \(\Pi_K\) with Lipschitz constant \(1\) to obtain
\begin{align}
\label{Eq:delta<L2norm-bc}
	\begin{aligned}
		\delta_\alpha^*(R X,X)^2&\le\delta_\alpha(R
		X,X)^2
		%
		\\
		&\le \norm{R_2X-Z }_{L^2(\Gamma)}^2 +
		\norm{R_1X-U}_{L^2(\Omega)}^2+\norm{R_1X-U}_{L^2(\Gamma)}^2,
	\end{aligned}
\end{align}
where \(R=(R_1,R_2): {S}_\ell^1(\mesh) \times {S}_\ell^1(\mesh) \to
H^1(\Omega)\times H^1(\Omega)\) is the auxiliary operator defined
in~\eqref{eq:Rell} and, as before, \(X=(U,Z)\) is the discrete solution
of~\eqref{mod-vred-bc-discretisation}. Combining their definitions
reveals the following orthogonality relationships: for all
\(\Phi_1,\Phi_2\in {S}_\ell^1(\mesh)\), we have
\begin{align}
\label{Eq:Ritz-bc}
\begin{multlined}
	\int_{\Omega} \nabla (R_1X-U)\cdot\nabla \Phi_2 + (R_1X-U)\Phi_2
\\
	=
	0
	=
	\int_{\Omega}\nabla (R_2X-U)\cdot\nabla \Phi_1 + (R_2X-U)\Phi_1.
\end{multlined}
\end{align}
In other words, \(U\) and $Z$ are, respectively the Ritz projections in $S_\ell^1(\mesh)$ of $R_1 X$ and $R_2 X$ with respect to the bilinear form \( (\cdot,\cdot)_{H^1(\Omega)} \). Consequently, 
similarly to the preceding section, we can quantify the available compactness by means of a duality argument thanks to the orthogonality \eqref{Eq:Galerkin-orthogonality-bc} of $\Res(X)$.

To this end, we restrict ourselves to polyhedral convex domains
\(\Omega\subset\R^d\) and  analyze the regularity of the solution of
the following Neumann problem: given $g=(g^\Omega,g^\Gamma) \in
L^2(\Omega) \times L^2(\Gamma)$, find $v_g \in H^1(\Omega)$ such that 
\begin{align*}
	-\Delta v_g+v_g=g^\Omega\quad\text{in}~\Omega\quad\text{and}\quad \partial_{\textsf{n}}v_g=g^\Gamma,
\end{align*}
which weakly reads as
\begin{align}
	\label{Eq:EllNeumann}
	\forall \vphi\in H^1(\Omega)
	\qquad
	\int_\Omega\nabla v_g\cdot\nabla\vphi + v_g\vphi
	=
	\int_\Omega g^\Omega \vphi +\int_\Gamma g^\Gamma \vphi.
\end{align}
Notice that the critical term on the right-hand side involves an
$L^2(\Gamma)$-trace of the test function $\vphi$.  Hence, in order to
precisely measure its regularity, we shall need a sharp trace theorem
for $L^2(\Gamma)$. Using fractional Sobolev spaces, the trace operator
is bounded as a map from $H^{\frac{1}{2}+\epsilon}(\Omega)$ to
$H^{\epsilon}(\Gamma)$ only for $\epsilon>0$, and therefore is not sharp
for $L^2(\Gamma)$. For a sharp trace theorem and thus avoiding
$\epsilon$, we invoke Besov spaces. Given $s>0$, $p,q \in [1,\infty]$,
we define the Besov space $B^s_q(L^p(\Omega))$ and its norm $\| \cdot
\|_{B^s_q(L^p(\Omega))}$ as in \cite[Section~2]{DeVore.Sharpley:1993}
through intrinsic moduli of smoothness.  Furthermore, we need  the
real interpolation method of Peetre based upon the so-called
$K$-functional; see, e.g., \cite{Brenner.Scott:2008}. Given two Banach
spaces $X_1$, $X_2$ with $X_1 \subset X_2$  and parameters $\theta \in
(0,1)$, $q \in [1,\infty]$, we denote its interpolation by
$(X_1,X_2)_{\theta,q}$ and its norm by $\| \cdot
\|_{(X_1,X_2)_{\theta,q}}$. 

\begin{prop}[Extra regularity for boundary control]
\label{P:Besov-reg-Neumann}
There is a constant $C_\Omega$ depending only on the convex domain $\Omega\subset\R^d$ such that, for any \( g=(g^\Omega, g^\Gamma) \in L^2(\Omega) \times L^2(\Gamma)\), the unique solution \(v_g \in H^1(\Omega)\) of~\eqref{Eq:EllNeumann} satisfies
\begin{align*}
	\|v_g\|_{B^{\frac32}_\infty(L^2(\Omega))}
	\le
	C_{\Omega} \left(
	 \|g^\Omega\|_{L^2(\Omega)}+\|g^\Gamma\|_{L^2(\Gamma)}
	\right).
\end{align*}
\end{prop}

\begin{proof}
\fbox{\scriptsize 1} We start by measuring the regularity of the right-hand side
\begin{equation*}
	\langle F_g, \vphi \rangle
	\defas
	\int_\Omega g^\Omega \vphi +\int_\Gamma g^\Gamma \vphi,
\quad
   \vphi \in H^1(\Omega),
\end{equation*}
in \eqref{Eq:EllNeumann}. In view of \cite[Section~4]{DeVore.Sharpley:1993} and \cite[Sections~1.2.5 and 1.3.4]{Triebel:1992}, the above definition of $B^s_q(L^p(\Omega))$ coincides with \cite[Definition~(2.52)]{Behrndt.Gesztesy.Mitrea:2022} in the sense of equivalent norms. Hence, we can use the sharp trace theorem \cite[Proposition~3.5]{Behrndt.Gesztesy.Mitrea:2022} to derive
\begin{align*}
	|\langle F_g, \vphi \rangle|
	&\leq
	\| g^\Omega \|_{L^2(\Omega)}  \| \vphi \|_{L^2(\Omega)} 
	+
    C \| g^\Gamma \|_{L^2(\Gamma)}  \| \vphi \|_{B^{\frac12}_1(L^2(\Omega))} 	
 \\
    &\leq
   	C \left(
   	\|g^\Omega\|_{L^2(\Omega)}+\|g^\Gamma\|_{L^2(\Gamma)}
   	\right)
   	\| \vphi \|_{B^{\frac12}_1(L^2(\Omega))} .
\end{align*}
Thanks to \cite{Johnen.Scherer:1977}, we have $B^{\frac12}_1(L^2(\Omega)) = (H^1(\Omega), L^2(\Omega))_{\frac12,1}$ and, in view of the duality theorem \cite[(14.1.8)]{Brenner.Scott:2008}, $(H^1(\Omega), L^2(\Omega))_{\frac12,1}^* = (L^2(\Omega), H^1(\Omega)^*)_{\frac12,\infty}$. Consequently,
\begin{equation}
\label{Eq:reg-rhs-Neumann}
	\| F_g \|_{(L^2(\Omega), H^1(\Omega)^*))_{\frac12,\infty}}
	\leq
	C \left(
	\|g^\Omega\|_{L^2(\Omega)}+\|g^\Gamma\|_{L^2(\Gamma)}
	\right).	
\end{equation} 

\fbox{\scriptsize 2} We next specify the corresponding regularity gain in the solution $v_g$. Replacing the right-hand side of \eqref{Eq:EllNeumann} by a generic functional $G \in H^1(\Omega)^*$, we readily observe 
\begin{equation*}
	\| v_G \|_{H^1(\Omega)}
	\leq
	\| G \|_{H^1(\Omega)^*}
\end{equation*}
for the corresponding solution $v_G$. Next, let us consider $G \in L^2(\Omega)$ and notice that this corresponds to a homogeneous Neumann problem with source term in $L^2(\Omega)$, i.e.\ \eqref{Eq:EllNeumann} with $g^\Omega = G$ and $g^\Gamma=0$. Since $\Omega$ is convex, we then have
\begin{equation*}
	\| v_G \|_{H^2(\Omega)}
	\leq
	C \| G \|_{L^2(\Omega)};
\end{equation*}
cf.\ \cite[Theorem~3.2.1.3]{Grisvard:1985}. Interpolating these two
inequalities with \cite[(14.1.5)]{Brenner.Scott:2008} gives 
\begin{equation*}
	\| v_G \|_{B^{\frac{3}{2}}_\infty(L^2(\Omega))}
   \leq
   C \| v_G \|_{(H^2(\Omega), H^1(\Omega))_{\frac{1}{2},\infty}}
	\leq
	C \| G \|_{(L^2(\Omega), H^1(\Omega)^*))_{\frac12,\infty}},
\end{equation*}
where the first inequality follows from $(H^2(\Omega), H^1(\Omega))_{\frac{1}{2},\infty} = B^{\frac{3}{2}}_\infty(L^2(\Omega))$, see \cite[6.2.4]{Bergh.Loefstroem:1976}, again taking into account \cite[Section~4]{DeVore.Sharpley:1993} and \cite[Sections~1.2.5 and 1.3.4]{Triebel:1992}. Hence, inserting \eqref{Eq:reg-rhs-Neumann} in the last inequality with $G=F_g$ finishes the proof.
\end{proof}

Proposition~\ref{P:Besov-reg-Neumann} puts us in the position to prove the following bound, which, thanks to the inequality \eqref{Eq:delta<L2norm-bc}, yields a bound  for the compact part of the error.

\begin{lem}[Upper bound for compact error -- boundary control]
\label{L:delta<hRes;Besov}
Let \(\Omega\subset \R^d\) be a convex domain with polyhedral boundary. Then the $L^2$-errors in \eqref{Eq:delta<L2norm-bc}  are bounded in terms of the residual of $X$:
  \begin{multline*}
    \|R_2X-Z\|_{L^2(\Gamma)}^2+\|R_1X-U\|_{L^2(\Gamma)}^2+\|R_1X-U\|_{L^2(\Omega)}^2
    \\
    \le C_{\Omega}^2
    C_\mesh^2\sum_{z\in\vertices}h_z \|\Res(X)\|^2_{H_z^*}.
  \end{multline*}
  Here \(C_\mesh, C_{\Omega}\) are essentially the constants
  from~\eqref{Eq:Res-localisation-bc} and Proposition~\ref{P:Besov-reg-Neumann}, respectively.
\end{lem}

\begin{proof}
We provide only a sketch of the proof, which is very similar to the one of Lemma~\ref{L:apost-deltaK} but involves some additional technicality due to the Besov regularity in Proposition~\ref{P:Besov-reg-Neumann}. We start with the 
term \(\|R_1X-U\|_{L^2(\Gamma)}^2+\|R_1X-U\|_{L^2(\Omega)}^2\). According to Proposition~\ref{P:Besov-reg-Neumann},  there exists \(\psi\in B^{\frac{3}{2}}_\infty(L^2(\Omega))\) weakly solving
\begin{gather*}
    -\Delta \psi+\psi=R_1X-U_h~\quad\text{in}~\Omega\qquad
    \text{and}\qquad
    \partial_{\textsf{n}}\psi=R_1X-U_h\quad\text{on}~\Gamma
\intertext{and}
    \| \psi \|_{B^{\frac{3}{2}}(L^2(\Omega))}
    \le
    C_{\Omega} (\|R_1X-U_h\|_{L^2(\Omega)}+\|R_1X-U_h\|_{L^2(\Gamma)}).
\end{gather*}
Again,  \(\mathcal{I}_{\textsf{sz}}:H^1(\Omega) \to {S}_\ell^1(\mesh)\) denots the Scott-Zhang quasi-interpolation operator~\cite{ScottZhang:1990}. Recalling~\eqref{Eq:EllNeumann} and the definition of $R$, we have thanks to the 
orthogonality~\eqref{Eq:Galerkin-orthogonality-bc} of $\Res(X)$ that
  \begin{align*}
    \|R_1X-U\|_{L^2(\Omega)}^2&+\|R_1X-U\|_{L^2(\Gamma)}^2=\int_\Omega \nabla (R_1X-U)\cdot\nabla
                                \psi+ (R_1X-U)\psi
    \\
    &=
                                \langle \Res_2(X),\psi\rangle
    =\sum_{z\in\vertices}\langle
                                \Res_2(X),(\psi-\mathcal{I}_{\textsf{sz}}
      \psi)\phi_z\rangle
    \\
                              &\le
                          \sum_{z\in\vertices}\|
                                \Res_i(X)\|_{H_z^*}\|\nabla((\psi-\mathcal{I}_{\textsf{sz}}
      \psi)\phi_z)\|_{L^2(\omega_z)},
  \end{align*}
where we have used that the Lagrange basis functions \(\{\phi_z:z\in\vertices\}\) of \(S_1^1(\mesh)\) form a partition of unity and that \(\operatorname{supp}{\phi_z}=\omega_z\), \(z\in\vertices\). Similarly to \eqref{Eq:countering-cutoff-H2}, we derive
\begin{equation*}
    \|\nabla((\psi-\mathcal{I}_{\textsf{sz}} \psi)\phi_z)\|_{L^2(\omega_z)}
    \le
   C_\mesh \|\psi \|_{H^1(\tilde \omega_z)}.
\end{equation*}
Interpolating with \cite[(14.1.5)]{Brenner.Scott:2008} both inequalities
yield
\begin{equation*}
	\|\nabla((\psi-\mathcal{I}_{\textsf{sz}} \psi)\phi_z)\|_{L^2(\omega_z)}
	\le
	C_\mesh h_z^{\frac{1}{2}}\|\psi \|_{B^{\frac{3}{2}}_\infty(L^2(\tilde \omega_z))}.
\end{equation*}
Employing averaged moduli, cf.\ \cite[Lemma~4.10]{Gaspoz.Morin:2013} and using that the overlapping of the domains \(\tilde \omega_z=\bigcup_{y\in \vertices\cap \omega_z}\omega_y\) is controlled by the shape regularity of \(\mesh\),
we conclude with~\eqref{Eq:regRes2} that
  \begin{multline*}
    \|R_1X-U\|_{L^2(\Omega)}^2+\|R_1X-U\|_{L^2(\Gamma)}^2
    \\
    \le
    C_\mesh C_{\Omega}\left(
     \sum_{z\in\vertices} h_z \| \Res_i(X) \|_{H_z^*}^2
    \right)^{\frac12}
    \left( 
     \|R_1X-U\|_{L^2(\Omega)}^2+\|R_1X-U\|_{L^2(\Gamma)}^2
    \right)^{\frac12}.
 \end{multline*}
 Applying similar arguments to \(\|R_2X-Z\|_{L^2(\Gamma)}^2\), the assertion follows.
\end{proof}

In combination with Theorem~\ref{T:compact-case}, we thus get the
following a posteriori bounds. 
\begin{thm}[Bounding the $d_\alpha$-error for boundary control -- compact case]
  \label{T:apost-Example-bcCC}
  Suppose in addition to the setting in
  Theorem~\ref{T:Apost1-ExampleBC} that \(\Omega\subset \R^d\) is
  convex with a polyhedral boundary. Then, we have
  \begin{multline*}
    \left(\frac1{d+1} \sum_{z\in\vertices} \|\Res(X)\|_{H_z^*}^2\right)^{\frac12}
    \le d_\alpha(x,X)
    \\\le
    C\left(
      1 + \frac{\kappa}{\sqrt\alpha}\,\left(
        \frac{ \sum_{z\in\vertices} h_z\|\Res(X)\|_{H_z^*}^2 }{ \sum_{z\in\vertices} \|\Res(X)\|_{H_z^*}^2}
      \right)^{\frac12}
    \right)
    \left( \sum_{z\in\vertices} \|\Res(X)\|_{H_z^*}^2 \right)^{\frac12}
  \end{multline*}
  The constant \( C \) depends on the Poincar\'e constant \(C_P\), the
  constant \(C_{\Omega}\) from Proposition~\ref{P:Besov-reg-Neumann} and
  the shape regularity of the mesh \(\mesh\); the constant \(\kappa\) is defined in~\eqref{kappa}.
\end{thm}

 In the absence of control constraints, we have from Theorem~\ref{T:apostNCC} the following a~posteriori bounds for the combined $H^1(\Omega)$-errors of the states.
 \begin{thm}[Bounding the $\norm{\cdot}$-error for boundary control]
 \label{T:apost-Example-bcNCC}
 Suppose in addition to the setting in Theorem~\ref{T:apost-Example-bcCC} that \(a=-\infty\), i.e.\ no control constraints apply. Then we have 
 \begin{multline*}
    \max\left\{
     \frac{\sqrt{\alpha}}{\sqrt{\alpha}+(1+C_\Gamma)},\,
      1- C \kappa\, G \big(\Res(X) \big)\right\}\left(\frac1{d+1}
      \sum_{z\in\vertices}\|\Res(X)\|_{H_z^*}^2\right)^{\frac12}
      \\
      \le
   \|x-X\|_{H^1(\Omega)}
     \le
    C\left( 1 + \kappa\, G \big( \Res(X) \big)
        \right)
     \left(
      \sum_{z\in\vertices}\|\Res(X)\|_{H_z^*}^2
    \right)^{\frac12},
  \end{multline*}
  with
  \begin{align*}
    G\big( \Res(X) \big) =  \left(\frac{\sum_{z\in\vertices}h_z\|\Res(X)\|_{H_z^*}^2}{\sum_{z\in\vertices}\|\Res(X)\|_{H_z^*}^2}\right)^{\frac12}.
  \end{align*}
The constant \( C \) depends on the Poincar\'e constant \(C_P\), the
constant \(C_{\Omega}\) from Proposition~\ref{P:Besov-reg-Neumann} and
the shape regularity of the mesh \(\mesh\); the constant \(\kappa\) is defined in~\eqref{kappa}.
\end{thm}

In view of $\norm{x-X} \leq d_\alpha(x,X)$, the upper bound in Theorem~\ref{T:Apost1-ExampleBC} can be used also for the $\norm{\cdot}$-error. Note that the first option in the max in the lower bound of Theorem~\ref{T:apost-Example-bcNCC} does not involve compactness.  We thus have upper and lower bounds for $\norm{\cdot}$-error which hold in general, i.e.\ do not need compactness. This pair of general bounds can be compared with the bounds in Theorem~\ref{T:apost-Example-bcNCC} using compactness. For the numerical comparison of such pairs, we consider
\begin{equation}
  \label{neumann-control-example}
  \begin{multlined}
    \min_{(q,u)\in L^2(\Gamma) \times H^1( \Omega)}
    \frac12 \norm{ u-\ud }_{L^2(\Omega)}^2 + \int_{\Gamma} g_1 u  + \frac\alpha2 \norm{q}_{L^2(\Gamma)}^2 \\
    \quad\text{subject to}\quad
    -\Delta u + u=f ~\text{in}~\Omega \quad\text{and}\quad
    \partial_{\textsf{n}}{u}=g_2 +q~\text{on}~\Gamma,
  \end{multlined}
\end{equation}
where $\Omega$ is the convex domain that is meshed in Figure~\ref{F:meshes-neumann} and has an  internal maximum  angle $\frac{35\pi}{36}$ at the origin, $u= 0 $, $z=r^{\frac{36}{35}}\cos(\frac{36}{35} \theta)$, $\ud = -\sqrt{\alpha} z$, 
$g_1 = \frac{\partial{z}}{\partial n}$, and $g_2 = \frac{1}{\sqrt{\alpha}} z$. 

As in Section~\ref{sec:Example}, the numerical simulations are carried out with linear finite elements, a standard residual estimator scaled by the same procedure, and the adaptive mesh refinement is based on D\"orfler's marking strategy~\cite{Doerfler:1996} with parameter~\(0.6\).

Figure~\ref{F:rates-neumann-control-example} depicts the $\norm{\cdot}$-error and the aforementioned associated bounds, while Figure~\ref{F:meshes-neumann} gives an idea for the underlying adaptive mesh refinement. The following differences to the discussion of the bounds for the $d_\alpha$-error in Section~\ref{sec:Example} are noteworthy. Compactness is useful in both upper and lower bound. The upper bound with compactness is advantageous from the start. This is related with the fact that both upper bounds depend on $\alpha$ only through the factor $\kappa$ from \eqref{kappa}.  For the lower bounds of the $\norm{\cdot}$-error, we have a similar situation as for the upper bounds of the $d_\alpha$-error. Indeed, the lower bounds with compactness improves the general one only for fine meshes and the necessary fineness increases with decreasing $\alpha$. 

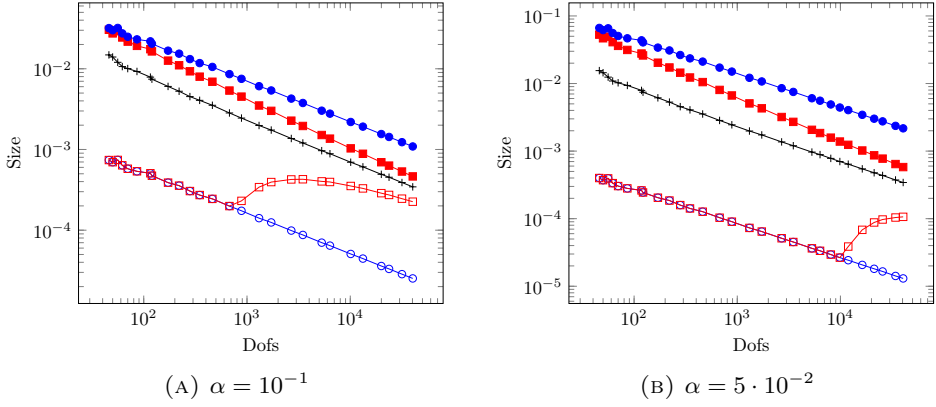
\begin{figure}[htb]
  \begin{subfigure}[b]{0.49\textwidth}
		\begin{tikzpicture}[scale=0.7]
			\begin{loglogaxis}[xlabel={Dofs}, ylabel={Size},legend pos=outer north east]
				\addplot[black, mark=+] table[x=Dofs , y= H1error] {./neumann-control-1p1.txt};
				\addlegendentry{$H^1$ error}
				\addplot[red, mark=square*] table[x=Dofs , y= Asbou] {./neumann-control-1p1.txt};
				\addlegendentry{as. upper }
				\addplot[blue, mark=*] table[x=Dofs , y= Nasbou] {./neumann-control-1p1.txt};
				\addlegendentry{dir. upper }
				\addplot[blue, mark=o] table[ x=Dofs, y=lextra] {./neumann-control-1p1.txt};
				\addlegendentry{lower}
                                \addplot[red, mark=square] table[ x=Dofs, y=lower] {./neumann-control-1p1.txt};
				\addlegendentry{clbd}
                                \legend{}
			\end{loglogaxis}
		\end{tikzpicture}
		\caption{ $\alpha = 10^{-1}$ }
	\end{subfigure}\hfill
	\begin{subfigure}[b]{0.49\textwidth}
			\begin{tikzpicture}[scale=0.7]
			\begin{loglogaxis}[xlabel={Dofs}, ylabel={Size}
                          ]
				\addplot[black, mark=+] table[x=Dofs , y= H1error] {./neumann-control5-2p1.txt};
				\addlegendentry{$H^1$-error}
				\addplot[red, mark=square*] table[x=Dofs , y= Asbou] {./neumann-control5-2p1.txt};
				\addlegendentry{cubd}
				\addplot[blue, mark=*] table[x=Dofs , y= Nasbou] {./neumann-control5-2p1.txt};
				\addlegendentry{gubd}
				\addplot[blue, mark=o] table[ x=Dofs, y=lextra] {./neumann-control5-2p1.txt};
				\addlegendentry{glbd}
                                \addplot[red, mark=square] table[ x=Dofs, y=lower] {./neumann-control5-2p1.txt};
				\addlegendentry{clbd}
                                \legend{}
			\end{loglogaxis}
		\end{tikzpicture}
		\caption{ $\alpha = 5\cdot 10^{-2}$ }
	\end{subfigure}
	\caption{$\norm{\cdot}$-error (+), general upper (\textcolor{blue}{$\bullet$}) and lower (\textcolor{blue}{$\circ$}) bound, as well as upper 
          (\textcolor{red}{$\textcolor{red}{\rule{.9ex}{.9ex}}$}) and lower
          bound with compactness versus DOFs for Example~\eqref{neumann-control-example}.
          \label{F:rates-neumann-control-example}}
	\end{figure}

         \begin{figure}[htb]
                      \includegraphics[width=.48\linewidth]{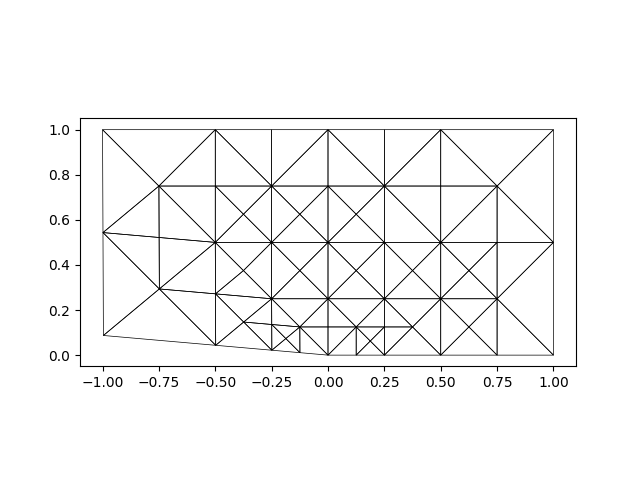}
                      \includegraphics[width=.48\linewidth]{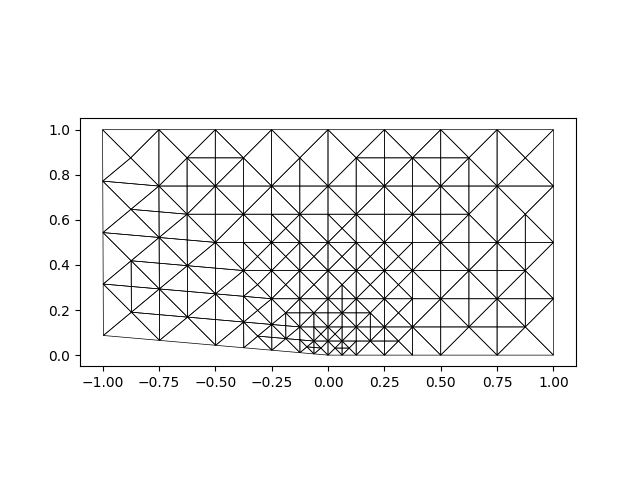}
                      \\
                      \includegraphics[width=.48\linewidth]{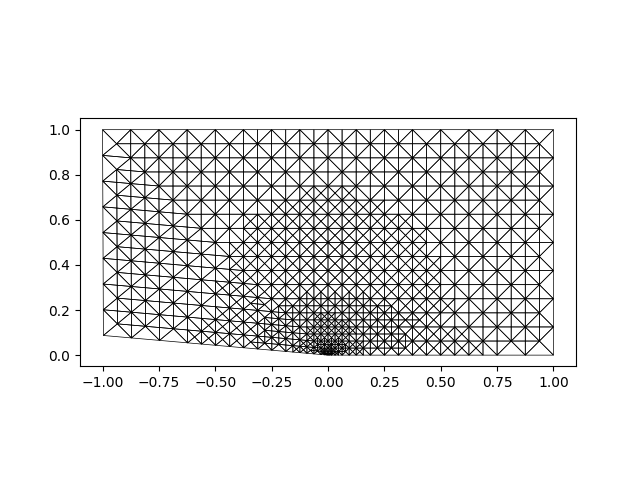}
                      \includegraphics[width=.48\linewidth]{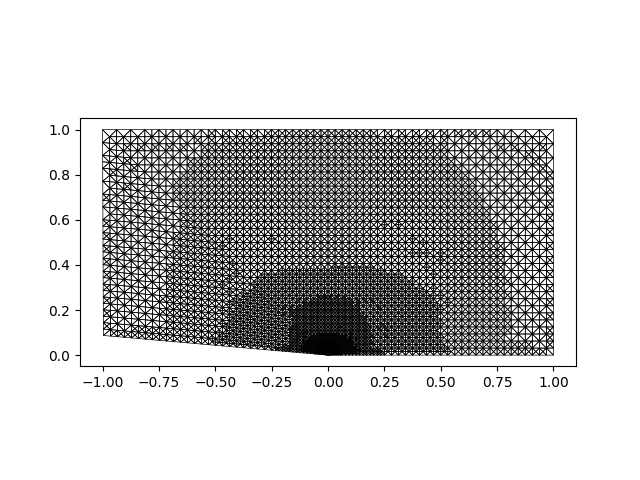}
                      \caption{Adaptive mesh refinement history for Example~\eqref{neumann-control-example} with $\alpha = 10^{-1}$.\label{F:meshes-neumann}}
                    \end{figure}

\subsection*{Acknowledgment}
Fernando Gaspoz is partially supported by Consejo Nacional de Investigaciones Científicas y Técnicas through grant PIP 11220200101180CO, by Agencia Nacional de Promoción Científica y Tecnológica through grants PICT-2020-SERIE A-03820 and 03267, and by Universidad Nacional del Litoral through grant CAI+D-2020 50620190100136LI.

\bibliographystyle{siam}
\bibliography{./lit}

\begin{thebibliography}{10}

\bibitem{AinsworthOden:2000}
{\sc M.~Ainsworth and J.~T. Oden}, {\em A posteriori error estimation in finite
  element analysis}, Pure and Applied Mathematics (New York),
  Wiley-Interscience [John Wiley \& Sons], New York, 2000.

\bibitem{Babuska:71}
{\sc I.~Babu\v{s}ka}, {\em Error-bounds for finite element method}, Numer.
  Math., 16 (1971), pp.~322--333.

\bibitem{Bastian_etal-DUNE:2021}
{\sc P.~Bastian, M.~Blatt, A.~Dedner, N.-A. Dreier, C.~Engwer, R.~Fritze,
  C.~Gräser, C.~Grüninger, D.~Kempf, R.~Klöfkorn, M.~Ohlberger, and
  O.~Sander}, {\em The {Dune} framework: {Basic} concepts and recent
  developments}, Computers \& Mathematics with Applications, 81 (2021),
  pp.~75--112.

\bibitem{Behrndt.Gesztesy.Mitrea:2022}
{\sc J.~Behrndt, F.~Gesztesy, and M.~Mitrea}, {\em Sharp {Boundary} {Trace}
  {Theory} and {Schrödinger} {Operators} on {Bounded} {Lipschitz} {Domains}},
  Sept. 2022.
\newblock arXiv:2209.09230 [math].

\bibitem{Bergh.Loefstroem:1976}
{\sc J.~Bergh and J.~Löfström}, {\em Interpolation spaces: an introduction},
  no.~223 in Die {Grundlehren} der mathematischen {Wissenschaften} in
  {Einzeldarstellungen}, Springer, Berlin, 1976.
\newblock OCLC: 2373287.

\bibitem{BonCanNocVee:2024}
{\sc A.~Bonito, C.~Canuto, R.~H. Nochetto, and A.~Veeser}, {\em Adaptive finite
  element methods}, 2024.

\bibitem{Brenner.Scott:2008}
{\sc S.~C. Brenner and L.~R. Scott}, {\em The {Mathematical} {Theory} of
  {Finite} {Element} {Methods}}, vol.~15 of Texts in {Applied} {Mathematics},
  Springer New York, New York, NY, 2008.

\bibitem{DeVore.Sharpley:1993}
{\sc R.~A. DeVore and R.~C. Sharpley}, {\em Besov {Spaces} on {Domains} in
  {Rd}}, Transactions of the American Mathematical Society, 335 (1993),
  pp.~843--864.

\bibitem{Doerfler:1996}
{\sc W.~D\"orfler}, {\em A convergent adaptive algorithm for {P}oisson's
  equation}, SIAM J. Numer. Anal., 33 (1996), pp.~1106--1124.

\bibitem{GaKrVeWo:2020}
{\sc F.~Gaspoz, C.~Kreuzer, A.~Veeser, and W.~Wollner}, {\em Quasi-best
  approximation in optimization with {PDE} constraints}, Inverse Problems, 36
  (2020), pp.~014004, 29.

\bibitem{Gaspoz.Morin:2013}
{\sc F.~D. Gaspoz and P.~Morin}, {\em Approximation classes for adaptive higher
  order finite element approximation}, Mathematics of Computation, 83 (2013),
  pp.~2127--2160.

\bibitem{Grisvard:1985}
{\sc P.~Grisvard}, {\em Elliptic problems in nonsmooth domains}, no.~24 in
  Monographs and studies in mathematics, Pitman Advanced Pub. Program, Boston,
  1985.

\bibitem{Grisvard:2011}
{\sc P.~Grisvard}, {\em Elliptic problems in nonsmooth domains}, vol.~69 of
  Classics in Applied Mathematics, Society for Industrial and Applied
  Mathematics (SIAM), Philadelphia, PA, 2011.
\newblock Reprint of the 1985 original [MR0775683], With a foreword by Susanne
  C. Brenner.

\bibitem{Hinze:2005}
{\sc M.~Hinze}, {\em A variational discretization concept in control
  constrained optimization: The linear-quadratic case}, Comp. Optim. Appl., 30
  (2005), pp.~45--61.

\bibitem{Johnen.Scherer:1977}
{\sc H.~Johnen and K.~Scherer}, {\em On the equivalence of the {K}-functional
  and moduli of continuity and some applications}, in Constructive {Theory} of
  {Functions} of {Several} {Variables}, W.~Schempp and K.~Zeller, eds., Berlin,
  Heidelberg, 1977, Springer, pp.~119--140.

\bibitem{KohlsKreuzerRoeschSiebert:18}
{\sc K.~Kohls, C.~Kreuzer, A.~R\"{o}sch, and K.~G. Siebert}, {\em Convergence
  of adaptive finite elements for optimal control problems with control
  constraints}, North-West. Eur. J. Math., 4 (2018), pp.~157--184, i.

\bibitem{KohlsRoeschSiebert:2012}
{\sc K.~Kohls, A.~R\"{o}sch, and K.~G. Siebert}, {\em A posteriori error
  estimators for control constrained optimal control problems}, in Constrained
  optimization and optimal control for partial differential equations, vol.~160
  of Internat. Ser. Numer. Math., Birkh\"{a}user/Springer Basel AG, Basel,
  2012, pp.~431--443.

\bibitem{KohlsRoeschSiebert:2014}
\leavevmode\vrule height 2pt depth -1.6pt width 23pt, {\em A posteriori error
  analysis of optimal control problems with control constraints}, SIAM J.
  Control Optim., 52 (2014), pp.~1832--1861.

\bibitem{KreuzerVeeser:2021}
{\sc C.~Kreuzer and A.~Veeser}, {\em Oscillation in a posteriori error
  estimation}, Numer. Math., 148 (2021), pp.~43--78.

\bibitem{LakkisMakridakis:2006}
{\sc O.~Lakkis and C.~Makridakis}, {\em Elliptic reconstruction and aposteriori
  error estimates for fully discrete linear parabolic problems}, Math. Comp.,
  75 (2006), pp.~1627--1658.

\bibitem{Lions:1971}
{\sc J.-L. Lions}, {\em Optimal Control of Systems Governed by Partial
  Differential Equations}, Die Grundlehren der mathematischen Wissenschaften,
  Springer, Berlin -- Heidelberg -- New York, 1.~ed., 1971.

\bibitem{MakridakisNochetto:2003}
{\sc C.~Makridakis and R.~H. Nochetto}, {\em Elliptic reconstruction and a
  posteriori error estimates for parabolic problems}, SIAM J. Numer. Anal., 41
  (2003), pp.~1585--1594.

\bibitem{Necas:62}
{\sc J.~Ne{\v c}as}, {\em {Sur une m\'ethode pour resoudre les \'equations aux
  d\'eriv\'ees partielles du type elliptique, voisine de la variationnelle}},
  Ann. Sc. Norm. Super. Pisa, Sci. Fis. Mat., III. Ser., 16 (1962),
  pp.~305--326.

\bibitem{Schatz:1974}
{\sc A.~H. Schatz}, {\em An observation concerning {R}itz-{G}alerkin methods
  with indefinite bilinear forms}, Math. Comp., 28 (1974), pp.~959--962.

\bibitem{ScottZhang:1990}
{\sc L.~R. Scott and S.~Zhang}, {\em Finite element interpolation of nonsmooth
  functions satisfying boundary conditions}, Math. Comp., 54 (1990),
  pp.~483--493.

\bibitem{Triebel:1992}
{\sc H.~Triebel}, {\em Theory of {Function} {Spaces} {II}}, Springer Basel,
  Basel, 1992.

\bibitem{Troeltzsch:2010}
{\sc F.~Tr\"{o}ltzsch}, {\em Optimal control of partial differential
  equations}, vol.~112 of Graduate Studies in Mathematics, American
  Mathematical Society, Providence, RI, 2010.
\newblock Theory, methods and applications, Translated from the 2005 German
  original by J\"{u}rgen Sprekels.

\bibitem{Verfuerth:2013}
{\sc R.~Verf\"{u}rth}, {\em A posteriori error estimation techniques for finite
  element methods}, Numerical Mathematics and Scientific Computation, Oxford
  University Press, Oxford, 2013.

\bibitem{Wheeler:1973}
{\sc M.~F. Wheeler}, {\em A priori {$L\sb{2}$} error estimates for {G}alerkin
  approximations to parabolic partial differential equations}, SIAM J. Numer.
  Anal., 10 (1973), pp.~723--759.

\end{thebibliography}
\end{document}